\documentclass{amsart}
\usepackage{a4wide}
\usepackage{amsmath,amsthm,amssymb,color,hyperref,graphicx,caption,subcaption,enumerate,tikz,float}
\usepackage[all]{xy}

\graphicspath{{Images/}}

\theoremstyle{plain}
\newtheorem{proposition}{Proposition}[section]
\newtheorem{theorem}[proposition]{Theorem}

\newtheorem{lemma}[proposition]{Lemma}

\theoremstyle{remark} \theoremstyle{definition}
\newtheorem{ex}{Example}
\newtheorem{remark}[proposition]{Remark}
\newtheorem{definition}[proposition]{Definition}
\newtheorem*{conjecture}{Conjecture}

\newcommand{\C}{\mathbb{C}}
\newcommand{\R}{\mathbb{R}}
\newcommand{\Q}{\mathbb{Q}}
\newcommand{\Z}{\mathbb{Z}}
\newcommand{\N}{\mathbb{N}}
\newcommand{\V}[1]{\textbf{#1}}
\newcommand{\p}{\mathfrak{p}}
\newcommand{\abs}[1]{\left \lvert #1 \right \rvert}
\newcommand{\norm}[1]{\left \lVert #1 \right \rVert}
\newcommand{\T}{T_{\rm ext}^{-1}}

\title{The geometry of non-unit Pisot substitutions}

\author[M. Minervino]{Milton Minervino}
\address[M.~M.]{Chair of Mathematics and Statistics,
Department of Mathematics and Information Technology, University of
Leoben, Franz-Josef-Strasse 18, A-8700 Leoben, Austria}
\email{minervino@math.tugraz.at}
\author[J. M. Thuswaldner]{J\"org M. Thuswaldner}
\address[J.~M.~T.]{Chair of Mathematics and Statistics,
Department of Mathematics and Information Technology, University of
Leoben, Franz-Josef-Strasse 18, A-8700 Leoben, Austria}
\email{joerg.thuswaldner@unileoben.ac.at}

\thanks{The authors are supported by the Doctoral Program W1230 ``Discrete Mathematics'' granted by the Austrian Science Fund (FWF)}



\date{\today}

\begin{document}

\begin{abstract}
Let $\sigma$ be a non-unit Pisot substitution and let $\alpha$ be the associated Pisot number. It is known that one can associate certain fractal tiles, so-called \emph{Rauzy fractals}, with $\sigma$. In our setting, these fractals are subsets of a certain open subring of the ad\`ele ring $\mathbb{A}_{\mathbb{Q}(\alpha)}$.  We present several approaches on how to define Rauzy fractals and discuss the relations between them. In particular, we consider Rauzy fractals as the natural geometric objects of certain numeration systems, define them in terms of the one-dimensional realization of $\sigma$ and its dual (in the spirit of Arnoux and Ito), and view them as the dual of multi-component model sets for particular cut and project schemes. We also define stepped surfaces suited for non-unit Pisot substitutions. We provide basic topological and geometric properties of Rauzy fractals associated with non-unit Pisot substitutions, prove some tiling results for them, and provide relations to subshifts defined in terms of the periodic points of $\sigma$, to adic transformations, and a domain exchange. We illustrate our results by examples on two and three letter substitutions.
\end{abstract}

\maketitle

\begin{section}{Introduction}

In order to derive dynamical properties of unit Pisot substitutions, Arnoux and Ito~\cite{AI:00} associated fractal tiles to these substitutions. These so-called \emph{Rauzy fractals}, named in honor of the seminal paper by Rauzy~\cite{R:82} who first worked out a complete example, have been studied extensively in the meantime (see {\it e.g.} \cite{BK:06,BS:05,IR:06,KS:12,SAI:01,ST:09,SS:02}). Siegel~\cite{Si:03} defined these fractals also for Pisot substitutions that are not necessarily unit. As conjectured already by Rauzy~\cite{R:88}, this requires to extend the representation space for these sets by certain $\mathfrak{p}$-adic factors (Figure~\ref{ifig} shows an example of a tiling induced by such a Rauzy fractal). In his fundamental PhD thesis, Sing~\cite{Sin:06} studied various properties of these non-unit Rauzy fractals in the context of model sets and recently, Akiyama {\it et al.}~\cite{ABBS:08} investigated them in the context of beta-numeration (see also \cite{BS:07} for some number theoretic properties of these sets).

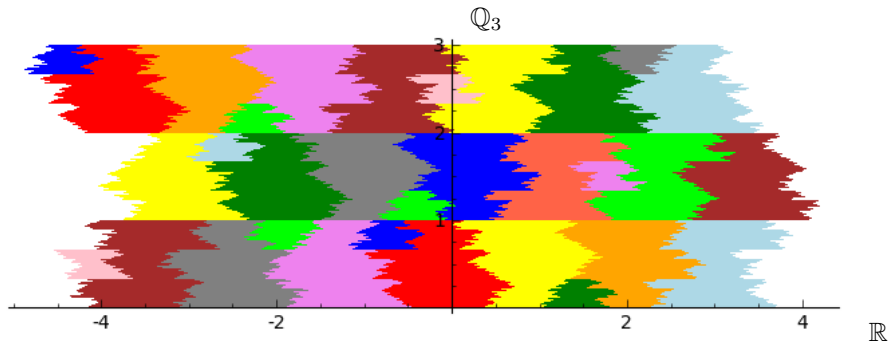
\begin{figure}[h]
\centering
\includegraphics[scale=.7]{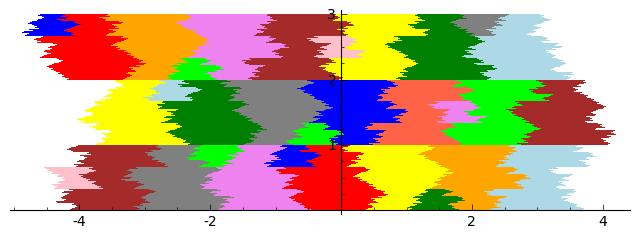}
\setlength\unitlength{1mm}
\begin{picture}(0,0)
\put(0,0){$\mathbb{R}$}
\put(-53,42){$\mathbb{Q}_3$}
\end{picture}
\caption{Tiling of the representation space $K_\sigma$ associated with the $2$-letter non-unit substitution $\sigma(1)=1^52$, $\sigma(2)=1^3$.\label{ifig}}
\end{figure}

Rauzy fractals have been defined in different ways and were studied in different contexts. So far, most of these approaches have been considered only in the case of unit Pisot substitutions. In the first part of the present paper we present and correlate these approaches in the general case of a non-unit Pisot substitution $\sigma$. Here, if $\alpha$ is the Pisot number associated with $\sigma$, the Rauzy fractal associated with $\sigma$ is defined in a certain subring $K_\sigma$ of the ad\`ele ring $\mathbb{A}_{\mathbb{Q}(\alpha)}$, whose non-Archimedean factors are determined by the prime divisors of the principal ideal $(\alpha)$ of the ring of integers $\mathcal{O}$ of $\mathbb{Q}(\alpha)$. In particular, we define Rauzy fractals in the following contexts.

\begin{itemize}
\item We review Dumont-Thomas numeration ({\it cf.} \cite{DT:89}), which is a generalization of the well-known notion of beta-numeration, and view Rauzy fractals as the natural geometric objects related to this kind of numeration (called \emph{Dumont-Thomas subtiles} in this context).

\item We extend the geometric realization of a substitution and its dual which was first studied by Arnoux and Ito~\cite{AI:00} to the non-unit case and define Rauzy fractals as renormalized pieces of stepped hypersurfaces with $\mathfrak{p}$-adic factors (and call them $E_1^*$-\emph{subtiles} in this setting; see Figure~\ref{ifig2}).

\item We present Sing's~\cite{Sin:06} construction of Rauzy fractals via cut and project schemes and define them in terms of a graph directed iterated function system. In this framework Rauzy fractals occur as the \emph{dual prototiles} of the multi-component model set associated with this cut and project scheme.
\end{itemize}

We show how these different approaches are related, provide conjugacies of the underlying mappings, and prove that the respective Rauzy fractals are all the same (up to affine transformations).

In the second part of this paper we establish geometric and topological properties of  Rauzy fractals, some of which occur in Sing's Thesis \cite{Sin:06} in the context of model sets, some of them are new. In particular, carrying over the model set definition of Rauzy fractals to Dumont-Thomas numeration, we prove that Rauzy fractals can be regarded as the solution of a graph directed set equation governed by the \emph{prefix automaton} of $\sigma$. This set equation provides a natural subdivision of the subtiles of a Rauzy fractal and highlights its self-affine structure that is inherited from the underlying substitution. We discuss how Rauzy fractals are related to certain subshifts defined in terms of periodic points of the substitution $\sigma$ and relate adic transformations to domain exchanges of subpieces of Rauzy fractals.

We show that Rauzy fractals always admit a multiple tiling of the representation space $K_\sigma$. Moreover, extending results of \cite{ABBS:08} on non-unit beta numeration we prove a tiling criterion for Rauzy fractals. In particular, we show that Rauzy fractals admit a tiling of the representation space provided that the representations of the underlying Dumont-Thomas numeration obey a certain finiteness condition which is an extension of the well-known property (F) of beta-numeration (see \cite{FS:92}). An example for such a tiling is depicted in Figure~\ref{ifig}.

We illustrate the results of our paper with examples for a two as well as a three letter substitution.

\begin{figure}[h]
\includegraphics[scale=0.35]{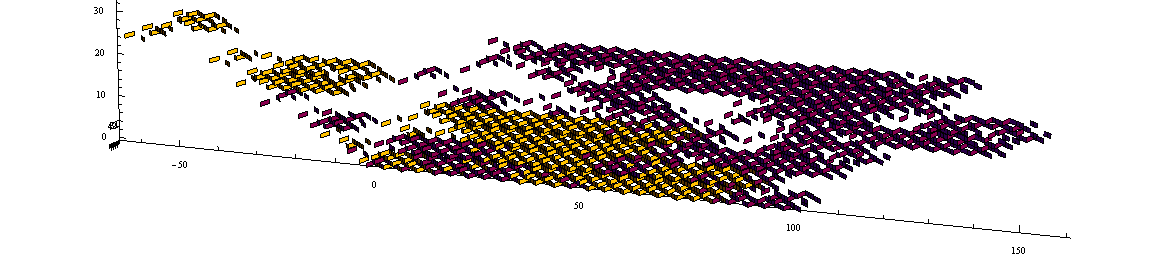}
\caption{$T_{\rm ext}^{-5}(\mathcal{U})$ for the substitution $\sigma(1)=2121^3$, $\sigma(2)=12$. \label{ifig2}}
\end{figure}

\end{section}

\begin{section}{Preliminaries}

\subsection{Pisot substitutions}
Let $\mathcal{A}=\{1,2,\ldots,n\}$ be a finite alphabet, and denote by $\mathcal{A}^*$ the set of finite words over $\mathcal{A}$. The set $\mathcal{A}^*$ endowed with the concatenation of words is a free monoid with the empty word $\epsilon$ as identity element. Given $w\in\mathcal{A}^*$ and $a\in\mathcal{A}$, let $\abs{w}$ be the length of the finite word $w$, $\abs{w}_a$ be the number of occurrences of $a$ in $w$. 
We denote by $\mathcal{A}^\omega$ the set of {\it(right) infinite words} and by $\phantom{}^\omega\mathcal{A}$ the set of {\it left-infinite words} over $\mathcal{A}$. The topology on $\mathcal{A}^\omega$ is the product topology of the discrete topology on $\mathcal{A}$. This implies that $\mathcal{A}^\omega$ is a compact Cantor set. A \emph{bi-infinite word} over $\mathcal{A}$ is a two-sided sequence in $\mathcal{A}^\mathbb{Z}$. We can equip $\mathcal{A}^\mathbb{Z}$ with a topology in an analogous way as we did for $\mathcal{A}^\omega$. A right or bi-infinite word $u$ is \emph{purely periodic} if there exists $v\in\mathcal{A}^*\setminus\{\epsilon\}$ such that $u=v^\omega$. Recall that  $u$ is \emph{uniformly recurrent} if every word occurring in $u$ occurs in an infinite number of positions with bounded gaps.

A \emph{substitution} is an endomorphism of the free monoid $\mathcal{A}^*$ with the condition that the image of each letter is non-empty and, for at least one letter $a\in\mathcal{A}$, $\abs{\sigma^k(a)}\rightarrow\infty$. A substitution naturally extends to the set of infinite and bi-infinite sequences. A one-sided (two-sided) \emph{periodic point} of $\sigma$ is an infinite (bi-infinite) word $u$ that satisfies $\sigma^k(u)=u$, for some $k>0$. If $k=1$, then $u$ is called \emph{fixed point} of $\sigma$.

We can naturally associate with a substitution $\sigma$ an \emph{incidence matrix} $M_\sigma$ with entries $(M_\sigma)_{a,b}=\abs{\sigma(b)}_a$, for all $a,b\in\mathcal{A}$.
The map $\V{P}:\mathcal{A}^*\rightarrow \N^n$, $w\mapsto (\abs{w}_1,\ldots,\abs{w}_n)^t$ is called the \emph{abelianisation map}. Obviously, we have $M_\sigma \circ \V{P} = \V{P} \circ \sigma$.  A substitution is \emph{primitive} if $M_\sigma$ is a primitive matrix. Every primitive substitution $\sigma$ has at least one periodic point and without loss of generality we can assume that $\sigma$ has at least one fixed point. Indeed, if $k$ is the period length then we may just work with $\sigma^k$ instead of $\sigma$.

The \emph{prefix automaton} associated to the substitution $\sigma$ is the directed graph with set of vertices $\mathcal{A}$ and set of labeled edges $a\xrightarrow{p}b$ if there exist $p,s\in\mathcal{A}^*$ such that $\sigma(a)=pbs$. 

Recall that an algebraic integer $\alpha>1$ is a \emph{Pisot number} if all its algebraic conjugates $\alpha'$ other than $\alpha$ itself satisfy $\abs{\alpha'}<1$. A substitution $\sigma$ is a \emph{Pisot substitution} if the dominant eigenvalue of $M_\sigma$ is a Pisot number. Furthermore we say that a Pisot substitution is \emph{unit} if $\alpha$ is a unit, otherwise we call it \emph{non-unit}. We say that a primitive substitution $\sigma$ is \emph{irreducible} if the characteristic polynomial of $M_\sigma$ is irreducible over $\mathbb{Q}$.
It is not hard to show that each irreducible Pisot substitution is primitive (see \emph{e.g.} \cite{CS:01}).

We introduce the following important combinatorial condition on substitutions.

\begin{definition}\label{defscc}
A substitution $\sigma$ over the alphabet $\mathcal{A}$ satisfies the \emph{strong coincidence condition} if for every pair $(b_1,b_2)\in\mathcal{A}^2$, there
exists $k\in\mathbb{N}$ and $a\in\mathcal{A}$ such that $\sigma^k(b_1)=p_1 a s_1$ and $\sigma^k(b_2)=p_2 a s_2$ with $\V{P}(p_1)=\V{P}(p_2)$ or $\V{P}(s_1)=\V{P}(s_2)$.
\end{definition}

\subsection{Substitution dynamical systems}

Recall that the \emph{two-sided shift} $S:\mathcal{A}^\mathbb{Z}\rightarrow \mathcal{A}^\mathbb{Z}$ is a homeomorphism on $\mathcal{A}^\mathbb{Z}$ and let $u \in \mathcal{A}^\mathbb{Z}$. The set $X_u:=\overline{\{S^j u : j\in\mathbb{N}\}}$
is a shift space, and we call it the \emph{symbolic dynamical system associated with} $u$.
If $\sigma$ is a primitive substitution, then all $\sigma$-periodic words are uniformly recurrent, and thus have the same language. The \emph{symbolic dynamical system associated with} $\sigma$ is the system associated with any of these. We will denote it by $(X_\sigma,S)$.
It is known that $(X_\sigma,S)$ is minimal and uniquely ergodic (see \cite{PF:02,Qu:10}).
Every word in $X_\sigma$ has a unique decomposition $w=S^k(\sigma(v))$, with $v\in X_\sigma$ and $0\leq k<\abs{\sigma(v_0)}$. This means that any word in $X_\sigma$ can be uniquely written in the form
\[ w=\ldots\mid \underbrace{\cdots}_{\sigma(v_{-1})}\mid
\underbrace{w_{-k}\cdots w_{-1}.w_0\cdots w_l}_{\sigma(v_0)}\mid
\underbrace{\cdots}_{\sigma(v_{1})}\mid  \underbrace{\cdots}_{\sigma(v_{2})}\mid
\cdots \]
with $\cdots v_{-1}v_0v_1\cdots \in\ X_\sigma$. Let $p=w_{-k}\cdots w_{-1}$ the prefix of $\sigma(v_0)$ of length $k$ and let $s=w_1\cdots w_l$ its suffix. The word $w$ is completely defined by the word $v$ and the decomposition of $\sigma(v_0)$ of the form $pw_0s$. Let $\mathcal{P}$ be the finite set of all such decompositions, {\it i.e.},
\begin{equation}\label{Pe}
\mathcal{P}=\{(p,a,s)\in \mathcal{A}^*\times \mathcal{A}\times \mathcal{A}^* : \exists\;b\in \mathcal{A}, \sigma(b)=pas\}\text{.}
\end{equation}

We can define a \emph{desubstitution map} $\chi$ on $X_\sigma$ (which sends $w$ to $v$), and a partition map $\rho$ from $X_\sigma$ to $\mathcal{P}$, corresponding to the decomposition of $\sigma(v_0)$. These two maps are continuous.
\begin{align*}
\chi:X_\sigma\rightarrow X_\sigma, w\mapsto v\quad & \text{such that }w=S^k\sigma(v), \text{and }0\leq k<\abs{\sigma(v_0)}\text{,} \\
\rho:X_\sigma\rightarrow \mathcal{P}, w\mapsto (p,w_0,s)\quad & \text{such that }\sigma(v_0)=pw_0s, \text{and }k=\abs{p}\text{.}
\end{align*}

Let $X_\mathcal{P}^l$ be the set of left-infinite sequences $(p_i,a_i,s_i)_{i\geq 0}=\cdots(p_1,a_1,s_1)(p_0,a_0,s_0)\in\phantom{}^\omega\mathcal{P}$ such that $\sigma(a_{i+1})=p_ia_is_i$, for all $i\geq 0$. The subshift $X_\mathcal{P}^l$ is sofic. The \emph{prefix-suffix development} is the map $E_\mathcal{P}:X_\sigma \rightarrow X_\mathcal{P}^l$ defined by
$E_\mathcal{P}(w)=(\rho(\chi^i(w)))_{i\geq 0}=(p_i,a_i,s_i)_{i\geq 0}$. If an infinite number of prefixes and suffixes are non-empty then we have the combinatorial expansion
$w=\lim_{k\rightarrow\infty}\sigma^k(p_k)\cdots \sigma(p_1)p_0.w_0s_0\sigma(s_1)\cdots\sigma^k(s_k)$, where the triples $(p_i,a_i,s_i)$ play the role of digits. It is shown in \cite{CS:01a} that the map $E_\mathcal{P}$ is continuous and onto $X_\mathcal{P}^l$, and it is one-to-one except on the orbit of periodic points of $\sigma$, where it is $k$-to-one with $k > 1$. 

If we project each of the $(p_i,a_i,s_i)$ of an element of $X_\mathcal{P}^l$ on the first component we obtain the labels of a left-infinite walk in the prefix automaton of the substitution $\sigma$. We will see in Section~\ref{sec:DT} that the hierarchical desubstitution structure of $X_\sigma$ is reflected in a natural way in the theory of Dumont-Thomas expansions.

\subsection{Algebraic framework}
We recall some notions from algebraic number theory. For more details we refer {\it e.g.} to \cite{N:99}.

Let $K$ be a number field and $\mathcal{O}$ its ring of integers. A \emph{prime} (or \emph{place}) $\mathfrak{p}$ is a class of equivalent valuations of $K$. An Archimedean equivalence class is called an infinite prime and it will be denoted by $\mathfrak{p}\mid\infty$, while a non-Archimedean equivalence class is called a finite prime, denoted by $\mathfrak{p}\nmid\infty$. The set of all infinite primes will be denoted by $S_\infty$. The infinite primes $\mathfrak{p}\mid\infty$ are obtained from the Galois embeddings $\tau:K\rightarrow\C$, and we define $\abs{\;\cdot\;}_\mathfrak{p}:K\rightarrow\R$ by $\abs{x}_\mathfrak{p}=\abs{\tau x}$ if $\mathfrak{p}$ is real, $\abs{x}_\mathfrak{p}=\abs{\tau x}^2$ if $\mathfrak{p}$ is complex. If $\mathfrak{p}$ is finite, we define $\abs{\;\cdot\;}_\mathfrak{p}:K\rightarrow\R$ by $\abs{x}_\mathfrak{p}=\mathfrak{N}(\mathfrak{p})^{-v_\mathfrak{p}(x)}$, where $\mathfrak{N}(\mathfrak{p})=p^{f_{\mathfrak{p}\mid (p)}}$ is the norm of the ideal $\mathfrak{p}$ lying over $(p)$, $f_{\mathfrak{p}\mid (p)}$ its inertia degree, $v_\mathfrak{p}:K^*\rightarrow\Z$ the $\mathfrak{p}$-adic valuation, defined by the unique prime ideal factorization
\[
(x):=x\mathcal{O}=\prod_\mathfrak{p}\mathfrak{p}^{v_\mathfrak{p}(x)}\text{.}
\]
Note that $v_\p(x)=0$ for almost all $\p$. Furthermore, given any $x\in K^*$, we have the important product formula \begin{equation}\label{product} \prod_{\p}\abs{x}_\mathfrak{p}=1\text{.} \end{equation}

We write $K_\mathfrak{p}$ for the completion of $K$ with respect to $\abs{\;\cdot\;}_\mathfrak{p}$. Denote by $\mathcal{O}_\mathfrak{p}$ the ring of integers of the completion $K_\mathfrak{p}$, \emph{i.e.}, the valuation ring
\[
\mathcal{O}_\mathfrak{p}=\{x\in K_\mathfrak{p}: v_\mathfrak{p}(x)\geq 0\}\text{.}
\]
By Ostrowski's theorem every $K_\mathfrak{p}$ is either isomorphic to $\R$ or $\C$, for infinite primes $\p$, or to a finite extension of $\Q_p$, for finite primes $\p$. Furthermore, for $\p$ finite, we can express any element of $K_\p$ as $\sum_{i=m}^\infty d_i\nu^i$, where $m\in\Z$, $\nu$ is a uniformiser, \emph{i.e.}, $v_\mathfrak{p}(\nu)=1$, and the $d_i$ are taken in a fixed system of representatives of the residue class field $\mathcal{O}_\p/\p\mathcal{O}_\p$.

For $\mathfrak{p}\mid\infty$ we equip $K_\mathfrak{p}$ with the real Lebesgue measure in case $\mathfrak{p}$ is real and with the complex Lebesgue measure otherwise. If $\mathfrak{p}\nmid\infty$ we equip $K_\mathfrak{p}$ with the Haar measure $\mu_\mathfrak{p}(x+\mathfrak{p}^k)=\mathfrak{N}(\mathfrak{p})^{-k}$.
We know (see for instance \cite[Proposition~2, Chapter~II]{Se:79}) that for every measurable subset $M$ of a local field $K_\mathfrak{p}$ and for every $x\in K_\mathfrak{p}$, 
\begin{equation}\label{serre}
\mu_\mathfrak{p}(x\cdot M)=\abs{x}_\mathfrak{p}\mu_\mathfrak{p}(M)\text{.}
\end{equation}
Consider the ad\`ele ring $\mathbb{A}_K$ of $K$ and the open subrings
\[ \mathbb{A}_{K,S}=\prod_{\mathfrak{p}\in S}K_\mathfrak{p}\times\prod_{\mathfrak{p}\notin S}\mathcal{O}_\mathfrak{p}\text{,} \]
for some finite set of primes $S$ containing $S_\infty$. Let $\mathcal{O}_S=\{x\in K:\abs{x}_\mathfrak{p}\leq 1\text{ for all }\mathfrak{p}\notin S\}$
be the set of $S$-integers. We know that
$\mathbb{A}_K=K+\mathbb{A}_{K,S_\infty}$ and $K\cap\mathbb{A}_{S_\infty}=\mathcal{O}_{S_\infty}$.
Moreover, $K$ is a discrete, co-compact subgroup of $\mathbb{A}_K$.

\end{section}

\begin{section}{The representation space}
Let $\sigma$ be an irreducible Pisot substitution with incidence matrix $M_\sigma$ and Pisot root $\alpha$.
Consider the number field $K=\Q(\alpha)$ of degree $n$ and signature $(r,s)$, and let $K_\infty=K \otimes_\Q \R\cong\R^r\times\C^s$.
Define the locally compact ring
\[ K_\alpha=K_\infty \times \prod_{\mathfrak{p}\mid (\alpha)}K_\mathfrak{p}=
\prod_{\mathfrak{p}\in S_\alpha}K_\mathfrak{p}\text{,}\]
where $S_\alpha=\{\mathfrak{p}:\mathfrak{p}\mid\infty\text{ or }\mathfrak{p}\mid (\alpha) \}$. We have injective ring homomorphisms $\Phi$ and $\Phi_\infty $ which diagonally embed $K$ into $K_\alpha$ and $K_\infty$, respectively. More precisely
\begin{align*}
\Phi:K &\longrightarrow K_\alpha\text{, }\quad \xi \phantom{a} \longmapsto
\prod_{\mathfrak{p}\in S_\alpha}\xi\text{,} \\
\Phi_\infty:K &\longrightarrow K_\infty\text{, }\quad \xi\phantom{a} \longmapsto
\prod_{\mathfrak{p}\in S_\infty}\xi\text{.}
\end{align*}
Note that $K$ acts multiplicatively on $K_\alpha$ by
$\xi\cdot(z_{\mathfrak{p}})_{\mathfrak{p}\in S_\alpha}=(\xi z_{\mathfrak{p}})_{\mathfrak{p}\in S_\alpha}$, for $\xi\in K$.

The \emph{representation space} is defined as
\[
K_\sigma=\prod_{\p\in S_\alpha\setminus\{\p_1\}}K_\p\text{,}
\]
where $\p_1$ is the infinite prime satisfying $\abs{\alpha}_{\p_1}=\alpha$. We equip $K_\alpha$ and $K_\sigma$ with the product of the metrics induced by $\abs{\;\cdot\;}_\p$ and with the product measure $\mu_{K_\alpha}$, respectively $\mu$, of the measures $\mu_\mathfrak{p}$, for $\mathfrak{p}\in S_\alpha$, respectively $\mathfrak{p}\in S_\alpha\setminus\{\p_1\}$. Notice that multiplication by $\alpha$ acts as a uniform contraction in $K_\sigma$. Let $\pi:K_\alpha\rightarrow K_\sigma$ be the projection which modules out the $\p_1$-coordinate, and let $\pi_{\p_1}:K_\alpha\to K_{\p_1}$ be the projection defined by $\pi_{\p_1}((z_\p)_\p)=z_{\p_1}$. We will denote by $\Phi'=\pi\circ\Phi$ the diagonal embedding of $K$ into $K_\sigma$. 

In case the substitution $\sigma$ is unit the representation space will consist only of Archimedean completions, because in this case $(\alpha)=\mathcal{O}$ and there is no finite prime $\p$ satisfying $\p\mid(\alpha)$.

\begin{lemma}
Let $M\subset K_\sigma$ be a measurable set. Then
\[ \mu(\alpha\cdot M)=\alpha^{-1}\cdot\mu(M)\text{.} \]
\end{lemma}
\begin{proof}
By (\ref{serre}) we obtain
\[  \mu(\alpha\cdot M)=\prod_{\mathfrak{p}\in S_\alpha\setminus\{\p_1\}}\abs{\alpha}_\mathfrak{p}\cdot\mu(M)\text{,} \]
and using the product formula (\ref{product}) we easily deduce that $\prod_{\mathfrak{p}\in S_\alpha\setminus\{\p_1\}}\abs{\alpha}_\mathfrak{p}=\alpha^{-1}$.
\end{proof}

Recall that a subset $W\subset K_\alpha$ is called a \emph{Delone set} if it is \emph{uniformly discrete} and \emph{relatively dense}. This means that there are radii $r,R>0$ so that each ball of radius $r$ (respectively $R$) contains at most (respectively at least) one point of $W$.

\begin{lemma}\label{Del}
The set $\Phi(\mathcal{O}_{S_\alpha})$ is a Delone set in $K_\alpha$.
\end{lemma}
\begin{proof}
The subring $\mathbb{A}_{K,S_\alpha}$ intersects the uniformly discrete subring $K$ in $\mathcal{O}_{S_\alpha}$, so $\Phi(\mathcal{O}_{S_\alpha})$ is likewise uniformly discrete in the closed subring $K_\alpha$. In order to show the relative denseness, note that $\mathbb{A}_{K,S_\alpha}$ is clopen in $\mathbb{A}_K$, so $\mathbb{A}_{K,S_\alpha}/\mathcal{O}_{S_\alpha}$ (with its quotient topology) is a clopen
subgroup of the compact group $\mathbb{A}_K/K$ and, hence, is compact. As $K_\alpha/\Phi(\mathcal{O}_{S_\alpha})$ is a
quotient of $\mathbb{A}_{K,S_\alpha}/\mathcal{O}_{S_\alpha}$, it is also compact.
\end{proof}

Let $\textbf{v}_\alpha=(v_1,\ldots,v_n)$ be a left eigenvector of $M_\sigma$ associated to $\alpha$, assume that $\textbf{v}_\alpha$ is scaled in a way that $v_i \in \mathbb{Q}(\alpha)$, and consider the $\Z$-module $V=\langle v_1,\ldots,v_n\rangle_\Z$.

\begin{lemma}
$V$ is a free $\mathbb{Z}$-module of rank $n$. Consequently, the numbers $v_i$ are rationally independent.
\end{lemma}
\begin{proof}
As $\V{v}_\alpha$ is an eigenvector, we get that $\alpha V\subset V$. Moreover,  $\V{v}_\alpha\not=\mathbf{0}$ which implies that $V\neq \{0\}$. Thus, since $\alpha$ is irrational of degree $n$, the elements  $v_j,\alpha v_j,\ldots,\alpha^{n-1}v_j \in V$ are linearly independent over $\mathbb{Q}$. Therefore
$\langle v_j,\alpha v_j,\ldots,\alpha^{n-1}v_j\rangle_\Z \subset V$
and, hence, $V$ has rank $n$.
\end{proof}

Since $V$ is an abelian group and $\alpha V\subset V$ we have that $V$ is a finitely generated $\Z[\alpha^{-1}]$-module.  As $v_i \in \mathbb{Q}(\alpha)$ there is $q\in\Z$ such that $v_i\in q^{-1}\Z[\alpha]$ holds for each $i\in\{1,\ldots,n\}$. Thus each generator of $V$ is taken in $q^{-1}\Z[\alpha]$ and, hence, $V\cdot\Z[\alpha^{-1}]$ is a (fractional) ideal of the ring $\Z[\alpha^{-1}]$.

\begin{lemma}\label{finind}
The following assertions hold:
\begin{enumerate}
\item $\mathcal{O}_{S_\alpha}=\mathcal{O}[\alpha^{-1}]$.
\item $V\cdot\Z[\alpha^{-1}]$ is a subgroup of finite index of $\mathcal{O}_{S_\alpha}$.
\end{enumerate}
\end{lemma}
\begin{proof}
Since $\alpha^{-1}\in\mathcal{O}_{S_\alpha}$ and $\mathcal{O}\subseteq\mathcal{O}_{S_\alpha}$ the inclusion $\mathcal{O}_{S_\alpha} \supseteq \mathcal{O}[\alpha^{-1}]$ follows. To prove the reverse inclusion, choose $x\in\mathcal{O}_{S_\alpha}$ and let $\mathfrak{p}\mid (\alpha)$. Then there exists $k\in\N$ such that $\abs{\alpha^k x}_\mathfrak{p}\leq 1$. Since $S_\alpha$ is a finite set of primes, setting
$h=\max\{k\in\N:\abs{\alpha^k x}_\mathfrak{p}\leq 1, \text{ for }\p\mid(\alpha) \}$ we get $\alpha^h x\in\mathcal{O}$, and, hence,  $\mathcal{O}_{S_\alpha} \subseteq \mathcal{O}[\alpha^{-1}]$.

$V$ is a subgroup of finite index of $q^{-1}\mathcal{O}$, for some $q\in\Z$, which implies that there exists $m\in\N$ such that $m q^{-1}\mathcal{O}\subseteq V$. Thus
$V\cdot\Z[\alpha^{-1}]\subseteq q^{-1}\mathcal{O}[\alpha^{-1}]\subseteq \frac{1}{m}V\cdot\Z[\alpha^{-1}]$ and it suffices to show that  $V\cdot\Z[\alpha^{-1}]$ is a subgroup of finite index of $\frac{1}{m}V\cdot\Z[\alpha^{-1}]$.
Suppose on the contrary that $mV\cdot\Z[\alpha^{-1}]$ is a subgroup of $V\cdot\Z[\alpha^{-1}]$ of infinite index, in particular $\abs{V\cdot\Z[\alpha^{-1}]/mV\cdot\Z[\alpha^{-1}]}>m^n$. Let $x_1,\ldots,x_{m^n+1}$ be $m^n+1$ pairwise different representatives of $V\cdot\Z[\alpha^{-1}]/mV\cdot\Z[\alpha^{-1}]$.
Since $x_1,\ldots,x_{m^n+1}\in V\cdot\Z[\alpha^{-1}]$, there exists $l\in\N$ such that
$x_1,\ldots,x_{m^n+1}\in V\langle 1,\alpha^{-1},\ldots,\alpha^{-l}\rangle_\Z$. As $V\langle 1,\alpha^{-1},\ldots,\alpha^{-l}\rangle_\Z$ is a $\Z$-module of rank at most $n$, $V\langle 1,\alpha^{-1},\ldots,\alpha^{-l}\rangle_\Z/mV\langle 1,\alpha^{-1},\ldots,\alpha^{-l}\rangle_\Z$ has index at most $m^n$, which implies that there exist $i,j\in\{1,\ldots,m^n+1\}$ such that $x_i\equiv x_j\bmod mV\langle 1,\alpha^{-1},\ldots,\alpha^{-l}\rangle_\Z$, contradicting $x_i\not\equiv x_j\bmod mV\cdot\Z[\alpha^{-1}]$.
\end{proof}

\begin{lemma}\label{delalpha}
$\Phi(V\cdot\Z[\alpha^{-1}])$ is a Delone set in $K_\alpha$.
\end{lemma}
\begin{proof}
This follows immediately by Lemma \ref{Del} and Lemma \ref{finind}.
\end{proof}
Thus $\Phi(V\cdot\Z[\alpha^{-1}])$ is a discrete subgroup of $K_\alpha$ and, hence, a lattice.
We look now for a fundamental domain of $K_\alpha$ modulo the lattice $\Phi(V\cdot\Z[\alpha^{-1}])$. We define $d_\p=\min\{v_\p(x):x\in V\}$, for every $\p\mid(\alpha)$.

\begin{lemma}\label{fd}
The set
\[
D=\left\{\sum_{i=1}^n r_i\Phi_\infty(v_i) : r_i\in[0,1)\right\}\times\prod_{\mathfrak{p}\mid(\alpha)}\mathfrak{p}^{d_\mathfrak{p}}
\]
is a fundamental domain for $K_\alpha \bmod \Phi(V\cdot\Z[\alpha^{-1}])$.
\end{lemma}

\begin{proof}
Let $w_1,\ldots,w_n$ be an integral basis of $\mathcal{O}$ over $\Z$. We claim that the set
\[
D_0:=\left\{\sum_{i=1}^n r_i\Phi_\infty(w_i): r_i\in[0,1)\right\}\times\prod_{\mathfrak{p}\mid(\alpha)}\mathcal{O}_\mathfrak{p}
\]
is a fundamental domain for $K_\alpha \bmod \Phi(\mathcal{O}_{S_\alpha})$.

To prove this claim let $\mathbf{z}:=(z_\mathfrak{p})_{\mathfrak{p}\in S_\alpha}\in K_\alpha$.
We know that $\Phi_\infty(w_1),\ldots,\Phi_\infty(w_n)$ is a basis of the real vector space $K_\infty$. Thus $\mathbf{z}_\infty:=(z_\p)_{\p\mid\infty}=\sum_{i=1}^n r_i\Phi_\infty(w_i)\in K_\infty$ for some $r_i\in\R$, and we denote by $\iota(\mathbf{z}_\infty)$ the element $\sum_{i=1}^n \lfloor r_i\rfloor w_i\in\mathcal{O}$.
For $\mathfrak{p}\mid(\alpha)$, $z_\mathfrak{p}\in K_\mathfrak{p}$ can be written as
\[
z_\mathfrak{p}=\sum_{i=-m}^{-1} c_i\nu^i + \sum_{i=0}^{\infty} c_i\nu^i\text{, }\quad m\in\N\text{,}
\]
where $\nu$ is a uniformiser and the $c_i$ are taken in a system of representatives of the residue class field $\mathcal{O}_\mathfrak{p}/\mathfrak{p}\mathcal{O}_\mathfrak{p}$.
Basically we view $z_\mathfrak{p}$ as the sum of a $\mathfrak{p}$-adic fractional part, that we denote by $\lambda_\p(z_\mathfrak{p})$, and a $\mathfrak{p}$-adic integral part.
Define
\[
y=\sum_{\mathfrak{p}\mid (\alpha)} \lambda_\p(z_\mathfrak{p})+\iota\Bigg( \mathbf{z}_\infty-
\Phi_\infty\bigg(\sum_{\mathfrak{p}\mid (\alpha)}\lambda_\p(z_\mathfrak{p})\bigg)\Bigg)\text{.}
\]
One checks that $y\in\mathcal{O}_{S_\alpha}$ and
$\V{z}-\Phi(y)\in D_0$.
Indeed, $\mathbf{z}_\infty-\Phi_\infty(y)$ is an element of the form $\sum_{i=1}^n r_i\Phi_\infty(w_i)$ with $r_i\in[0,1)$, by definition of $y$, and, for $\mathfrak{p}\mid(\alpha)$,
observe that both $z_\mathfrak{p}-\Phi_\p(\sum_{\mathfrak{p}\mid (\alpha)}\lambda_\p(z_\mathfrak{p}))$ and $\Phi_\mathfrak{p}(\iota(\mathbf{z}_\infty-\Phi_\infty(\sum_{\mathfrak{p}\mid (\alpha)}\lambda_\p(z_\mathfrak{p}))))$ are in $\mathcal{O}_\mathfrak{p}$. Furthermore $\V{z}-\Phi(x)\notin D_0$ for all $x\in \mathcal{O}_{S_\alpha}\setminus\{y\}$ (note that the intervals for the $a_i$ in the definition of $D_0$ are half-open).

Now we replace the lattice $\Phi(\mathcal{O}_{S_\alpha})$ with the sublattice $\Phi(V\cdot\Z[\alpha^{-1}])$. As $w_1,\ldots,w_n$ is a $\Q$-basis for $K$, the same holds for $v_1,\ldots,v_n$. The completion of $V$ at $\p$, \emph{i.e.}, $V_\p:=V\otimes_{\Z}\Z_p$ is isomorphic to $\mathfrak{p}^{d_\mathfrak{p}}$.
We can express an element $z_\p$ of the completion $K_\p$ as $z_\mathfrak{p}=\sum_{i=-m}^{-1} c_i\nu^i + \sum_{i=0}^{\infty} c_i\nu^i$ where $\nu$ is a uniformiser and the $c_i$ are taken in a set of representatives of the residue class field $V_\mathfrak{p}/\mathfrak{p}V_\mathfrak{p}$ isomorphic to $\p^{d_\p}/\p^{d_\p+1}$. Thus we can adapt all the arguments given above to get a unique element $y\in V\cdot\Z[\alpha^{-1}]$ such that $\V{z}-\Phi(y)\in D$.
\end{proof}

\end{section}

\begin{section}{The geometry of Dumont-Thomas numeration}\label{sec:DT}

\subsection{Dumont-Thomas numeration}
Dumont and Thomas~\cite{DT:89} studied numeration systems associated with a primitive substitution $\sigma$. This notion of numeration allows to expand real numbers with respect to a real base $\alpha > 1$, which is the Perron-Frobenius eigenvalue of the substitution. Dumont-Thomas expansions depend on the prefix automaton of the substitution and on the left eigenvector $\V{v}_\alpha$ associated with $\alpha$. The digit set for the expansions is $\mathcal{D}=\{\delta(p) : (p,a,s)\in\mathcal{P}\}$, where $\mathcal{P}$ is defined in \eqref{Pe} and $\delta$ is the map given by
\[
\delta:\mathcal{A}^*\to \Q(\alpha)\text{,}\quad \delta(p)=\langle \textbf{P}(p),\V{v}_\alpha\rangle\text{.}
\]
Notice that $\mathcal{D}\subset V$ is finite and depends on the normalization of $\V{v}_\alpha$.
A sequence $(\delta(p_i))_{i\geq 1}\in\mathcal{D}^\omega$ is called $(\sigma,a)$-\emph{admissible} if there exists a walk in the prefix automaton labeled by $(p_i)_{i\geq 1}$ starting from $a$ with infinitely many non-empty suffixes.

\begin{proposition}[Dumont-Thomas~\cite{DT:89}]\label{DT}
Let $\sigma$ be a primitive substitution on the alphabet $\mathcal{A}$ and fix $a\in\mathcal{A}$.
For every $x\in[0,\delta(a))$, there exists a unique $(\sigma,a)$-admissible sequence
$(\delta(p_i))_{i\geq 1}\in\mathcal{D}^\omega$ such that
\begin{equation}\label{eq:exp} x=\sum_{i\geq 1}\delta(p_i)\alpha^{-i}\text{.} \end{equation}
\end{proposition}
We will call an expansion of this form a \emph{$(\sigma,a)$-expansion} and we will denote its sequence of digits by $(x)_{\sigma,a}$.

For $x\in\R^+$, let $m\geq -1$ be the smallest integer such that $\alpha^{-m-1}x\in[0,\delta(a))$, for some $a\in\mathcal{A}$. 
Then we can expand $\alpha^{-m-1}x$ using Proposition~\ref{DT} and obtain a walk labeled by $(p_{-i})_{i\geq -m}$ such that
\[
x=\sum_{i=0}^m \delta(p_{m-i}) \alpha^{m-i} +\sum_{j\geq 1} \delta(p_{-j}) \alpha^{-j}\quad\Longleftrightarrow\quad (x)_{\sigma,a}=\delta(p_{m})\cdots\delta(p_0).\delta(p_{-1})\delta(p_{-2})\cdots
\]
Set $X=\bigcup_{a\in\mathcal{A}} \big([0,\delta(a)) \times \{a\}\big)$  and define the map
\[ T_\sigma: X \to X \text{,}\quad (y,b) \mapsto \big(\alpha y - \delta(p), a\big)\text{,} \]
where $a\in\mathcal{A}$ and $p\in\mathcal{A}^*$ are uniquely determined by $\sigma(b)=p a s$ and $\alpha y-\delta(p)\in[0,\delta(a))$. 
Given any $(y,b)\in X$ we get its $(\sigma,b)$-expansion by computing its $T_\sigma$-orbit. Notice that $T_\sigma$ is not injective and the pre-image has the form
\begin{equation}\label{T-1} T_\sigma^{-1}(x,a)=\bigcup_{b\stackrel{p}{\longrightarrow}a}\{(\alpha^{-1}(x+\delta(p)),b)\}\text{, }\quad \text{for }(x,a)\in X\text{.}
\end{equation}
It is easy to see from this identity that for $(x,a)\in X$ we have
\begin{equation}\label{trick}
\alpha^mT_\sigma^{-m}(x,a)=x+\alpha^mT_\sigma^{-m}(0,a)\text{, }\quad\forall m\in\N\text{,}
\end{equation}
where $x+(z,a)=(x+z,a)$ is used.

\subsection{Integers and fractional parts}\label{fracint}
We will be interested in \,\textquotedblleft integers\textquotedblright\, and \,\textquotedblleft fractional parts\textquotedblright\, obtained from the Dumont-Thomas numeration system generated by the substitution $\sigma$, \emph{i.e.}, all those $x\in\R^+$ such that only non-negative (respectively negative) powers of $\alpha$ occur in its $(\sigma,a)$-expansion, for some $a\in\mathcal{A}$. 

In the sequel, for a sequence of subsets $\{A_k\}_{k\in\N}$ of a topological space,  we write $\mathop{\rm Lim}_{k\to\infty} A_k$ for the \emph{topological limit}  
(if it exists; see \emph{e.g.} \cite[Chapter~II, \S29]{K:66} for the definition of this object).

Let $\Z_{\sigma,a}^{(k)}=\alpha^k\cdot T_\sigma^{-k}(0,a)\subset\R$ be the set of real numbers corresponding to all finite walks of length $k$ in the prefix automaton 
ending at state $a$, \emph{i.e.}, the sums $\sum_{i=0}^{k-1}\delta(p_i)\alpha^i$, 
where $a_k\stackrel{p_{k-1}}{\longrightarrow}\cdots\stackrel{p_1}{\longrightarrow} a_1\stackrel{p_0}{\longrightarrow}a$. To such an element we associate the left-sequence of digits $\delta(p_{k-1})\cdots \delta(p_1)\delta(p_0).\in \phantom{}^*\mathcal{D}$.
Observe that these sets are \textbf{not} nested (see Example \ref{ex1}).

\begin{definition}
The set of $(\sigma,a)$-\emph{integers} is the topological limit $\Z_{\sigma,a}= \mathop{\rm Lim}_{k\to\infty} \Z_{\sigma,a}^{(k)}$.  
We call the union of the $\Z_{\sigma,a}$  for $a\in\mathcal{A}$  $\sigma$-\emph{integers} and denote it by $\Z_\sigma$.
\end{definition}
The topological limit in the definition exists since for every interval $[0,\ell]\subset\R^+$ there exists $k_0\in\N$ such that 
$\mathop{\rm Lim}_{k\to\infty} \Z_{\sigma,a}^{(k)}\cap[0,\ell] = \Z_{\sigma,a}^{(k)}\cap[0,\ell]$, for each $k\geq k_0$.

Notice that the set of $\sigma$-integers is a subset of $V$ and is discrete and closed. 
The set $\Z_{\sigma,a}$ is the set of those finite sums $\sum_{i=0}^{k-1}\delta(p_i)\alpha^i\in \Z_{\sigma,a}^{(k)} $, $k\in\N$, 
whose associated sequence of digits can be left-padded by zeros. In particular, $\Z_{\sigma,a}\subsetneq \bigcup_{k\geq 0} \Z_{\sigma,a}^{(k)}$, that is, 
not every \emph{approximation} is a $(\sigma,a)$-integer.

\begin{definition} The set of $(\sigma,a)$-\emph{fractional parts} is defined as
\begin{equation}
\text{Frac}(\sigma,a)=V\cdot\Z[\alpha^{-1}]\cap[0,\delta(a))\text{,}
\end{equation}
and $\text{Frac}(\sigma)=\bigcup_{a\in\mathcal{A}}\text{Frac}(\sigma,a)=V\cdot\Z[\alpha^{-1}]\cap[0,\max_{a\in\mathcal{A}}\delta(a))$, will be called the set of $\sigma$-\emph{fractional parts}.
\end{definition}

An element $x\in\text{Frac}(\sigma,a)$ has a $(\sigma,a)$-expansion $(x)_{\sigma,a}=.\delta(p_{-1})\delta(p_{-2})\cdots$, where $(p_{-k})_{k\geq 1}$ is the label of an infinite walk in the prefix automaton starting at state $a$.

\begin{lemma}\label{le:T} $T_\sigma$ maps $\text{Frac}(\sigma)$ onto $\text{Frac}(\sigma)$.
\end{lemma}
\begin{proof}
Given $(x,a)\in \text{Frac}(\sigma,a)$, $T_\sigma(y,b)=(x,a)$ for all $(y,b)$ such that $y=\alpha^{-1}(x+\delta(p))$ and $\sigma(b)=pas$. 
Notice that there exists at least one $(y,b)$ of this form since the prefix automaton is strongly connected by the primitivity of $\sigma$. 
It is clear that $y\in V\cdot\Z[\alpha^{-1}]$. Furthermore, if $(x)_{\sigma,a}=.\delta(p_1)\delta(p_2)\cdots$, then $(y)_{\sigma,b}=.\delta(p)\delta(p_1)\delta(p_2)\cdots$ 
which implies that $y\in[0,\delta(b))$.
\end{proof}

\begin{ex}\label{ex1}
Let $\sigma$ be the substitution $\sigma(1)=121$, $\sigma(2)=11$.
We have \[ M_\sigma=\begin{pmatrix} 2 & 2 \\ 1 & 0 \end{pmatrix}\text{, }\qquad
\det(xI-M_\sigma)=x^2-2x-2\text{.} \]
This substitution is an irreducible non-unit Pisot substitution with associated Pisot root $\alpha=1+\sqrt{3}$.
A left eigenvector associated to $\alpha$ for $M_\sigma$ is $\V{v}_\alpha=(\frac{\alpha}{2},1)$. From the prefix automaton of the substitution depicted in Figure~\ref{presufpicture} we can see that the set of digits is $\mathcal{D}=\{\delta(\epsilon),\delta(1),\delta(12)\}$.
\begin{figure}[h]
\includegraphics[scale=0.5]{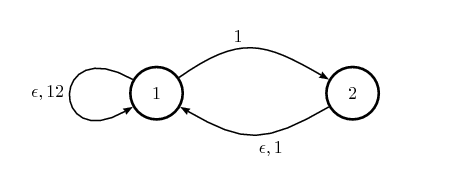}
\caption{The prefix automaton of the substitution $\sigma(1)=121$, $\sigma(2)=11$. \label{presufpicture}}
\end{figure}
In Figure~\ref{tsigmafigure} we illustrate the map $T_\sigma$ and its combinatorial structure.
If the letter $a$ of a point $(x,a)\in X$ is equal to $1$ then the function depicted in the left square is used to compute the first coordinate of its image, if it is $2$, we use the one in the right square. Moreover, the second coordinate of the image of $(x,a)$ is $1$ if the line of the linear piece is black, and $2$ if it is light gray. For example, given $(x,1)\in [1,\frac{\alpha-1}{2})\times\{1\}$, which is in the left square, after one iteration of $T_\sigma$ it will jump to the right square.
\begin{figure}
\begin{center}
\begin{tikzpicture}[scale=2]
\draw(0,0)node[below]{$0$}--(1.366,0)node[below]{$\frac{\alpha}{2}$}--(1.366,1.366)--(0,1.366)node[left]
{$\frac{\alpha}{2}$}--(0,0) (0,1)node[left]{$1$}--(1.366,1) (1/2,0)node[below]{$\frac{1}{2}$}--(1/2,1.366) (0.866,0)node[below]{$\frac{\alpha-1}{2}$}--(0.866,1.366);
\draw[very thick,black](0,0)--(1/2,1.366) (0.866,0)--(1.366,1.366);
\draw[very thick,lightgray](1/2,0)--(0.866,1);
\draw[dashed](0,0)--(1.366,1.366);
\end{tikzpicture}\qquad\qquad
\begin{tikzpicture}[scale=2]
\draw(0,0)node[below]{$0$}--(1.366,0)node[below]{$\frac{\alpha}{2}$}--(1.366,1.366)--(0,1.366)node[left]
{$\frac{\alpha}{2}$}--(0,0) (0,1)node[left]{$1$}--(1.366,1) (1/2,0)node[below]{$\frac{1}{2}$}--(1/2,1.366) (1,0)node[below]{$1$}--(1,1.366);
\draw[very thick,black](0,0)--(1/2,1.366) (1/2,0)--(1,1.366);
\draw[dashed](0,0)--(1.366,1.366);
\end{tikzpicture}
\end{center}
\caption{The map $T_\sigma$ \label{tsigmafigure}}
\end{figure}
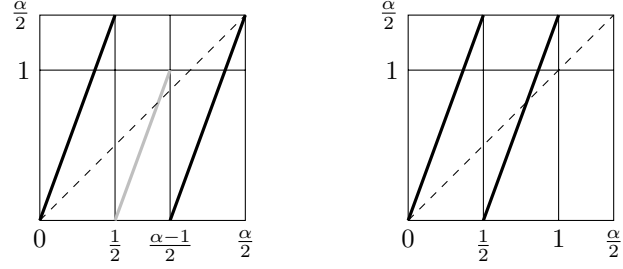
We compute as an example the orbit of $(\frac{1}{4},1)$:
\[ \xymatrix{ (\frac{1}{4},1)\ar[r]^{T_\sigma} & (\frac{\alpha}{4},1)\ar[r]^{T_\sigma} & (\frac{1}{2},2)\ar[r]^{T_\sigma} & *[l]{\,(0,1)\,} \ar@(dr,ur)[]_{T_\sigma} } \]
Thus we have $(\frac{1}{4})_{\sigma,1}=.0\delta(1)\delta(1)$. Observe that $(\frac{\alpha^3}{4},1)=(\tfrac{3\alpha}{2}+1,1)\in\Z_{\sigma,1}^{(k)}$ for all $k\geq 2$, thus $\tfrac{3\alpha}{2}+1$ is a $(\sigma,1)$-integer. On the other hand $(\tfrac{3\alpha}{2}+1,2)\in \Z_{\sigma,2}^{(2)}$, with associated walk $2\stackrel{\delta(1)}{\longrightarrow}1\stackrel{\delta(1)}{\longrightarrow}2$,
but $(\tfrac{3\alpha}{2}+1,2)\notin \Z_{\sigma,2}^{(3)}$, and this is due to the fact that we cannot left-pad its expansion by $0$s, \emph{i.e.}, we can extend the walk in the automaton on the left only by adding a digit $\delta(1)$. As another example we have $\delta(1)\delta(12).\in\Z_{\sigma,1}^{(2)}$, with associated walk $2\stackrel{\delta(1)}{\longrightarrow}1\stackrel{\delta(12)}{\longrightarrow}1$, but it cannot be extended to any infinite walk. In this sense, it remains an approximation. We list some other expansions:
\[
(\tfrac{\alpha-1}{3})_{\sigma,2}=.\overline{\delta(1)0}\qquad (\alpha-2)_{\sigma,1}=.\delta(1)\delta(1)\overline{0\delta(12)}\text{,}\qquad  (\tfrac{\alpha-1}{4})_{\sigma,1}=.0\delta(12)\delta(12)
\]

\end{ex}

\subsection{Tiles associated with Dumont-Thomas Numeration}
We are now in a position to define tiles for Dumont-Thomas numeration. These tiles form a generalization of the tiles defined in Akiyama~\cite{A:02} and Akiyama {\it et al.}~\cite{ABBS:08} in the context of beta-numeration.

\begin{definition}\label{subtiles}
The \emph{Dumont-Thomas subtiles} associated with an irreducible Pisot substitution $\sigma$ are defined as
\begin{equation}
\mathcal{R}_\sigma(a)=\overline{\Phi'(\Z_{\sigma,a})} \qquad (a\in \mathcal{A})\text{.}
\end{equation}
The \emph{Dumont-Thomas central tile} is defined as
\begin{equation}
\mathcal{R}_\sigma=\bigcup_{a\in\mathcal{A}}\mathcal{R}_\sigma(a)\text{.}
\end{equation}
\end{definition}
An example of a Dumont-Thomas central tile is depicted in Figure~\ref{fig:DT}.

For $\V{z}=(z_\p)_{\p\in S_\alpha\setminus\{\p_1\}}\in K_\sigma$ define the norm $\|\V{z}\|=\max\{|z_\p|_\p : \p\in S_\alpha\setminus\{\p_1\} \}$, and set
\begin{equation}\label{ball} M=\frac{\max\{\|\Phi'(\delta(p))\|:\delta(p)\in\mathcal{D}\}} {1-\|\Phi'(\alpha)\|}\text{.} \end{equation}
Note that  the Dumont-Thomas subtiles $\mathcal{R}_\sigma(a)$ are closed by definition. Furthermore they are contained in the closed ball 
$B(\V{0},M)=\{\V{z}\in K_\sigma:\|\V{z}\|\leq M\}$. Thus they are non-empty compact sets. We will prove more properties of these tiles later.

Again generalizing \cite{A:02} (see also \cite{ABBS:08} and \cite{BS:05})  for $x\in\text{Frac}(\sigma)$ we define the $x$-\emph{tiles} as
\[
\mathcal{R}_x=\bigcup_{\{a\in \mathcal{A}: \, x\in[0,\delta(a))\}}
\overline{\Phi'(\mathop{\rm Lim}_{k\to\infty}\alpha^k T_\sigma^{-k}(x,a))}\text{.}
\]
Using \eqref{trick} we easily see that they are unions of subtiles translated by $\Phi'(x)$ that depend on the number of basic intervals which contain $x$, {\it i.e.},
\[
\mathcal{R}_x=\bigcup_{\{a\in \mathcal{A}: \, x\in[0,\delta(a))\}}\mathcal{R}_\sigma(a)+\Phi'(x)\text{.}
\]
For this reason it suffices to consider subtiles in the sequel.
\footnotetext[1]{Ambiguity with colors and grayscale is due on whether you are reading an electronic version of the paper or a printed one.}
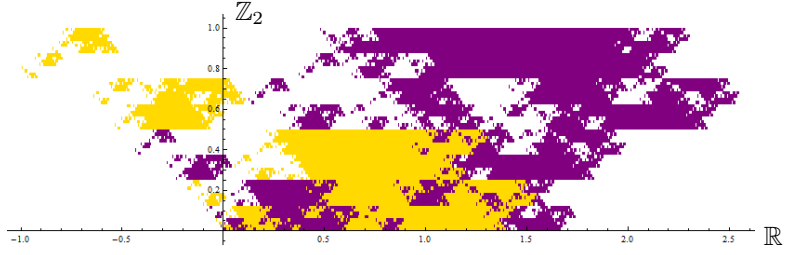
\begin{figure}[h!]
\includegraphics[scale=0.35]{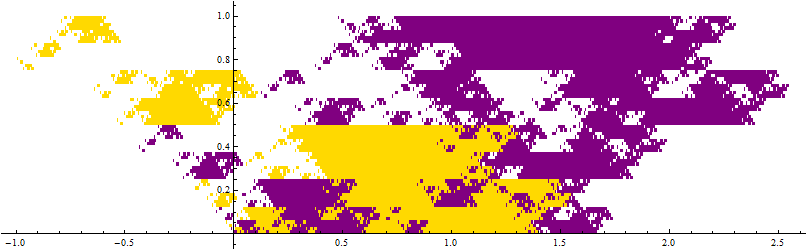}
\begin{picture}(0,0)
\put(0,0){$\mathbb{R}$}
\put(-200,85){$\mathbb{Z}_2$}
\end{picture}
\caption[Caption for LOF]{The central tile $\mathcal{R}_\sigma$ for $\sigma(1)=2121^3$, $\sigma(2)=12$ subdivided in the purple (dark gray) subtile 
$\mathcal{R}_\sigma(1)$ and the yellow (light gray) subtile $\mathcal{R}_\sigma(2)$.\footnotemark[1] \label{fig:DT} Here we represented each 
$\sum_{i=0}^\infty d_i\alpha^i\in\Z_2$ by $\sum_{i=0}^\infty d_i2^{-i-1}\in[0,1]$, with $d_i\in\{0,1\}$.}
\end{figure}

\begin{definition}[{see also \cite[Definition~6.80]{Sin:06}}] The \emph{stepped surface} for an irreducible Pisot substitution $\sigma$ with associated Pisot root $\alpha$ is
\begin{equation}\label{stepped} \mathcal{S}=\{(\Phi(x),a)\in K_\alpha\times\mathcal{A} : x\in \text{Frac}(\sigma,a) \}\text{.}  \end{equation}
The set of projected points of the stepped surface into $K_\sigma$ given by
\begin{equation} \label{GammaDef}
\Gamma=\pi(\mathcal{S})=\{(\Phi'(x),a)\in K_\sigma\times\mathcal{A} :
x\in \text{Frac}(\sigma,a) \}
\end{equation}
will be called the \emph{translation set}.
\end{definition}

As we will see later, this set is the natural translation set for a (multiple) tiling induced by the subtiles. Sing~ \cite[Definition~6.80]{Sin:06} defines stepped surfaces in the context of cut and project schemes; we will come back to this in Section~\ref{sec:MS}.

For our purposes (particularly in Section~\ref{Sec:Ex}) we will interpret $(\gamma,a)\in K_\alpha\times\mathcal{A}$ either as a colored translation vector or as a colored face of the fundamental domain $K_\alpha/\Phi(V\cdot\Z[\alpha^{-1}])$ (see Lemma \ref{fd}). To be more precise, in this latter case, $(\gamma,a)$ will be represented as $\gamma+D_a$, where
\[ D_a=\bigg\{\sum_{i\neq a} r_i\Phi_\infty(v_i) : r_i\in[0,1)\bigg\} \times\prod_{\mathfrak{p}\mid(\alpha)}\mathfrak{p}^{d_\mathfrak{p}}\text{,} \]
and $d_\mathfrak{p}=\min\{ v_\mathfrak{p}(x) : x\in V \}$.

The set function $T_\sigma^{-1}$ defined in (\ref{T-1}) is defined on subsets of $\mathbb{R}\times\mathcal{A}$. 
Thus $T_\sigma^{-1}$ cannot be extended to $K_\alpha\times\mathcal{A}$ in a natural way. 
However, its restriction to subsets of $\mathbb{Q}(\alpha)\times\mathcal{A}$ admits a natural extension to $K_\alpha\times\mathcal{A}$. 
We call this extension $T_{\rm ext}^{-1}$. Its precise definition is
\begin{equation} \label{Text}
\T: K_\alpha\times\mathcal{A}\longrightarrow 2^{K_\alpha\times\mathcal{A}}\text{, } \quad \T(\gamma,a)=\bigcup_{b\stackrel{p}{\longrightarrow} a} 
\{(\alpha^{-1}(\gamma+\Phi(\delta(p))),b)\}\text{.}
\end{equation}

We can iterate $\T$ for $m$ times and get
\begin{align}
(\T)^m(\gamma,a)=&\bigcup_{b_m\stackrel{p_{m-1}}{\rightarrow} \cdots\stackrel{p_1}{\rightarrow} b_1 \stackrel{p_0}{\rightarrow} a}
(\alpha^{-m}(\gamma+\Phi(\delta(p_0)+\alpha\delta(p_1)+\cdots+\alpha^{m-1}\delta(p_{m-1}))),b_m) \label{E1fold} \\
=&\bigcup_{\sigma^m(b)=pas}(\alpha^{-m}(\gamma+\Phi(\delta(p))),b)\text{.} \notag
\end{align}

\begin{proposition}\label{inv}
The set $\mathcal{S}$ is invariant under $\T$.
\end{proposition}

\begin{proof}
We prove the proposition in two steps:

\begin{itemize}
\item {\it If $(\Phi(x),a)\in\mathcal{S}$ then $\T(\Phi(x),a)\in\mathcal{S}$.} This is equivalent in showing that every element of $T_\sigma^{-1}(x,a)\in\text{Frac}(\sigma)$. 
But this is a direct consequence of Lemma \ref{le:T}.

\item {\it Distinct faces have disjoint images.}
Suppose $(\Phi(y),b)\in \T(\Phi(x_1),a_1)\cap \T(\Phi(x_2),a_2)$, that is \[ (y,b)\in T_\sigma^{-1}(x_1,a_1)\cap T_\sigma^{-1}(x_2,a_2)\text{.} \] 
This implies that $T_\sigma(y,b)=(x_1,a_1)$ and $T_\sigma(y,b)=(x_2,a_2)$, which is impossible unless $(x_1,a_1)=(x_2,a_2)$, since $y$ has a unique $(\sigma,b)$-expansion. \qedhere
\end{itemize}
\end{proof}

Observe that we can write $\mathcal{R}_\sigma(a)$ in terms of the extended mapping $\T$:
\begin{equation}\label{subnew}
\mathcal{R}_\sigma(a) = \overline{\Phi'(\mathop{\rm Lim}_{k\to\infty} \Z_{\sigma,a}^{(k)})} = \mathop{\rm Lim}_{k\to\infty} \Phi'(\Z_{\sigma,a}^{(k)}) =
\mathop{\rm lim_H}_{k\to\infty} \Phi'(\Z_{\sigma,a}^{(k)}) = \mathop{\rm lim_H}_{k\to\infty} \pi(\alpha^k T_{\rm ext}^{-k}(\mathbf{0},a))\text{,}
\end{equation}
where $\mathop{\rm lim_H}$ denotes the limit with respect to the Hausdorff metric. Indeed, the third equality holds since all $\Phi'(\Z_{\sigma,a}^{(k)})$ are contained in a compact set, and the fourth follows easily recalling the definition of $\Z_{\sigma,a}^{(k)}$ and observing that $\T\circ\Phi=\Phi\circ T_\sigma^{-1}$ and $\Phi'=\pi\circ\Phi$.

\begin{figure}[h]
\includegraphics[scale=0.4]{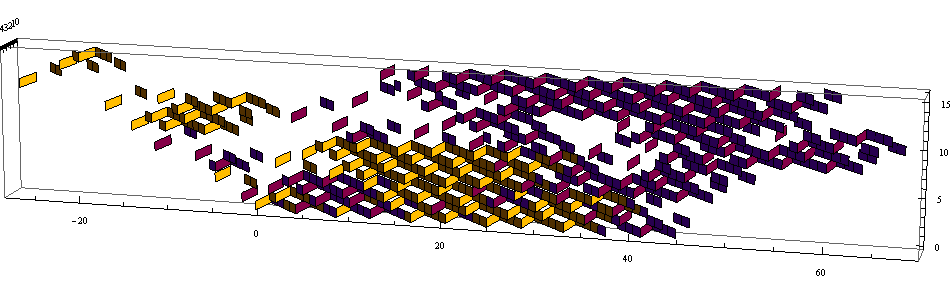}
\caption{$T_{\rm ext}^{-4}((\V{0},1)\cup(\V{0},2))$ for the substitution $\sigma(1)=2121^3$, $\sigma(2)=12$.}
\end{figure}
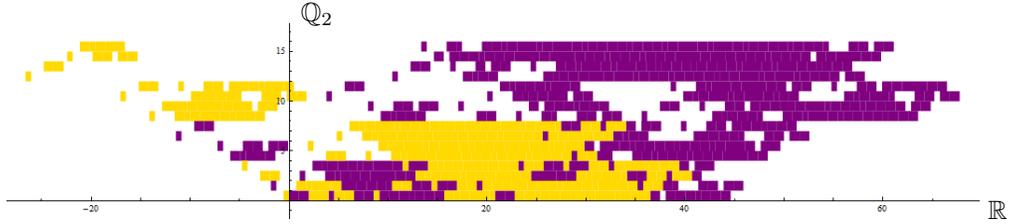
\begin{figure}[h]
\includegraphics[scale=0.35]{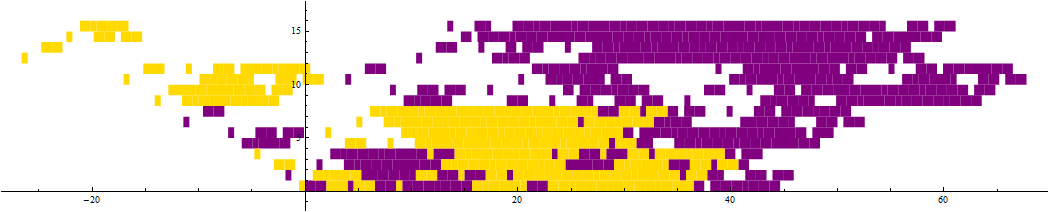}
\begin{picture}(0,0)
\put(0,0){$\mathbb{R}$}
\put(-260,75){$\mathbb{Q}_2$}
\end{picture}
\caption{$T_{\rm ext}^{-4}((\V{0},1)\cup(\V{0},2))$ projected into $K_\sigma$.}
\end{figure}

\subsection{Projective limit}
For $a\in\mathcal{A}$, let us consider again the sets $\Z_{\sigma,a}^{(i)}=\alpha^i\cdot T_\sigma^{-i}(0,a)$, $i\geq 1$, which can be thought as \emph{approximations} of the set of integers $\Z_{\sigma,a}$.
The sets $\Z_{\sigma,a}^{(i)}$ together with the continuous maps $f_i^j: \Z_{\sigma,a}^{(j)} \to \Z_{\sigma,a}^{(i)}$, $\sum_{k=0}^{j-1}\delta(p_k)\alpha^k \mapsto \sum_{k=0}^{i-1}\delta(p_k)\alpha^k$, for $i,j\in\N^+$ and $j\geq i$, form a projective system. Consider the projective limit
\begin{equation} \hat{\Z}_{\sigma,a}=\varprojlim_i \Z_{\sigma,a}^{(i)} = \Big\{ (x_i)_{i=1}^\infty\in\prod_{i\geq 1} \Z_{\sigma,a}^{(i)}: \text{ for all }j\geq i,\; f_i^j(x_j)=x_i \Big\}\text{,} \end{equation}
and the union $\hat{\Z}_\sigma=\bigcup_{a\in\mathcal{A}} \hat{\Z}_{\sigma,a}$.

If we give to each $\Z_{\sigma,a}^{(i)}$ the discrete topology and to $\prod_{i\geq 1} \Z_{\sigma,a}^{(i)}$ the product topology, the space $\hat{\Z}_{\sigma,a}$ inherits a topology which turns it into a compact space.

We can equip $\hat{\Z}_{\sigma,a}$ with two other topologies. Indeed, equip $\hat{\Z}_{\sigma,a}$ with the topology defined by the distance \[ d(x,y)=2^{-\max\{m\in\N:\; x_m=y_m\}}\text{, }\quad \text{for }x=(x_i)_{i\geq 1}\text{, }y=(y_i)_{i\geq 1}\in \hat{\Z}_{\sigma,a}\text{.}\] In this way $\hat{\Z}_{\sigma,a}$ is a compact Cantor set isomorphic to the subshift of finite type formed by the left-infinite sequences $(\delta(p_i))_{i\geq 0}\in\phantom{}^\omega\mathcal{D}$ such that $(p_i)_{i\geq 0}$ is the labeling of a left-infinite walk $\cdots\xrightarrow{p_2} a_2 \xrightarrow{p_1} a_1 \xrightarrow{p_0} a$ in the prefix automaton of $\sigma$.

On the other hand we can equip $\hat{\Z}_{\sigma,a}$ with the topology defined by the distance $d(x,y)=\norm{\Phi'(x)-\Phi'(y)}$. An element of $\hat{\Z}_{\sigma,a}$ can be represented as an infinite sum
$\sum_{i\geq 0}\delta(p_i)\alpha^i$, where each truncation $\sum_{i=0}^\ell\delta(p_i)\alpha^i$ is contained in $\Z_{\sigma,a}^{(\ell+1)}$. Extending $\Phi'$ continuously, we can use this mapping in order to map the infinite sums in $\hat{\Z}_{\sigma,a}$ to $K_\sigma$. The image  $\Phi'(\hat{\Z}_{\sigma,a})$ is exactly the Dumont-Thomas subtile $\mathcal{R}_\sigma(a)$.

The definition of the projective limit $\hat{\Z}_{\sigma,a}$ encompasses these two points of view, thus we can consider one of its elements either as an admissible left infinite sequence or as an infinite sum. Another advantage is that we have all the approximations (\emph{i.e.}, the truncations) of the elements included in this vision. In both interpretations multiplication by $\alpha$ acts as a contraction. We come back to $\hat{\Z}_{\sigma,a}$ in Section~\ref{sec-adic}.

\end{section}

\begin{section}{Geometric realizations of a substitution}

The aim of this section is to generalize the formalism of one-dimensional geometric realizations of substitutions and their dual (see Arnoux and Ito~\cite{AI:00}) to the non-unit case. This will lead to another definition of Rauzy fractal and stepped surface for a non-unit Pisot substitution.

\subsection{The maps $E_1$ and $E_1^*$}\label{subsec51}

A consequence of dealing with non-unit Pisot substitutions is that $M_\sigma:\mathbb{Z}^n\rightarrow\mathbb{Z}^n$ is not invertible.
We define a suitable bigger space where $M_\sigma$ acts and is invertible.

Indeed, let $Z=\bigcup_{k\geq 0}M_\sigma^{-k}\mathbb{Z}^n$. We denote by $\mathcal{F}$ the infinite dimensional real vector space of the maps $Z\times \mathcal{A}\rightarrow\R$ that take value zero except for a finite set.
For $\V{x}\in Z$, $a\in\mathcal{A}$ denote by $[\V{x},a]$ the element of $\mathcal{F}$ which takes value $1$ at $(\V{x},a)$ and $0$ elsewhere; the set $\{[\V{x},a] : \V{x}\in Z, a\in\mathcal{A}\}$ is a basis of $\mathcal{F}$. The support of an element of $\mathcal{F}$ is the set of $(\V{x},a)$ on which it is not zero.

We define the \emph{one-dimensional geometric realization} $E_1$ of $\sigma$ on $\mathcal{F}$ by
\[
E_1[\V{y},b]=\sum_{b \stackrel{p}{\longrightarrow} a}[M_\sigma \V{y}-\V{P}(p),a]\text{.}
\]
Denote by $\mathcal{F}^*$ the space of linear forms on $\mathcal{F}$ with finite support, \emph{i.e.}, those linear forms for which there exists a finite subset $X$ of $Z\times\mathcal{A}$ such that the form is $0$ on any element of $\mathcal{F}$ whose support does not intersect $X$; this space admits as basis the set $\{[\V{x},a]^* : \V{x}\in Z, a\in\mathcal{A}\}$. We can associate to $E_1$ its \emph{dual} map $E_1^*$ on $\mathcal{F}^*$.
\begin{proposition}
The dual map $E_1^*$ is defined on $\mathcal{F}^*$ by
\begin{equation}\label{geosub}
E_1^*[\mathbf{x},a]^*=\sum_{b\stackrel{p}{\longrightarrow} a}[M_\sigma^{-1} (\mathbf{x}+\mathbf{P}(p)),b]^*\text{.} \end{equation}
\end{proposition}
\begin{proof}
By definition of the dual map we have
\begin{align*}
\langle E_1^*[\V{x},a]^*,[\V{y},b] \rangle &=\langle [\V{x},a]^*,E_1[\V{y},b] \rangle \\
&=\langle [\V{x},a]^*, \sum_{b\stackrel{p}{\longrightarrow} c}[M_\sigma \V{y}-\V{P}(p),c] \rangle\text{.}
\end{align*}
This product can take only $0$ and $1$ as values and is not zero if and only if $c=a$ and $M_\sigma \V{y}-\V{P}(p)=\V{x}$. Since $M_\sigma$  is invertible as a map from $Z$ to $Z$, this implies $\V{y}=M_\sigma^{-1}(\V{x}+\V{P}(p))$.
\end{proof}

Denote by $\mathbb{H}$ the hyperplane of $\R^n$ orthogonal to $\V{v}_\alpha$, and let $\mathbb{H}^\geq$ be the set $\{\V{x}\in\R^n : \langle \V{x},\V{v}_\alpha\rangle\geq 0 \}$, \emph{i.e.}, the half-space above $\mathbb{H}$. The half-space $\mathbb{H}^<$ strictly below $\mathbb{H}$ is defined in the same fashion.
We look for all those $\V{x}\in Z$ that are close to the hyperplane $\mathbb{H}$, in particular, we want $\V{x}\in \mathbb{H}^\geq$ and $\V{x}-\V{e}_a\in \mathbb{H}^<$ to be true for some $a\in\mathcal{A}$.
One problem arises: we get too many points with this property and we would like to have a discrete set. Thus we enlarge again our space by adding the non-Archimedean completions in order to distribute the points with respect to their $\p$-adic height. Define $Z_\text{ext}=Z\times\prod_{\mathfrak{p}\mid(\alpha)}K_\mathfrak{p}$ and a map
\begin{align*}
\Psi:Z & \longrightarrow Z_\text{ext}\text{,} \\
\V{x} & \longmapsto (\V{x},\langle\V{x},\V{v}_\alpha\rangle,\ldots,\langle\V{x},\V{v}_\alpha\rangle)\text{.}
\end{align*}
As $M_\sigma^t$ acts as a multiplication by $\alpha$ along $\mathbf{v}_\alpha$ we can extend $M_\sigma$ naturally to $Z_\text{ext}$. Indeed, it acts as a
multiplication by $\alpha$ in each of the non-Archimedean factors. Without  risk of confusion this extension will again be denoted by $M_\sigma$.
Moreover, set $\mathbb{H}_\text{ext} = \mathbb{H} \times\prod_{\mathfrak{p}\mid(\alpha)}K_\mathfrak{p}$. This is the space where the subtiles will live.
We can also carry over the definition of $\mathcal{F}$ to the space $\mathcal{F}_{\text{ext}}$ of mappings from $\Psi(Z)\times\mathcal{A}$ to $\mathbb{R}$:
\[ 
\xymatrix{Z\times\mathcal{A} \ar[d]_{[\V{x},a]}\ar[r]^\Psi & \Psi(Z)\times\mathcal{A} \ar[dl]^{[\Psi(\V{x}),a]} \\ \R & } 
\]
Similarly we carry over the space $\mathcal{F}^*$ to $\mathcal{F}^*_\text{ext}$. Thus we can extend the maps $E_1$ and $E_1^*$ respectively on $\mathcal{F}_\text{ext}$ and $\mathcal{F}^*_\text{ext}$, denoting them again by the same names.

We can interpret geometrically an element $[\V{x},a]\in \mathcal{F}_\text{ext}$ as a segment $\{\V{x}-\theta\Psi(\V{e}_a) : \theta\in[0,1]\}$ in $Z_\text{ext}$.
In this way we get the \textquotedblleft broken line\textquotedblright\, with reversed orientation associated with the fixed point $u=u_0u_1\ldots$ of the substitution in terms of $E_1$:
\begin{equation}\label{LuStar}
-L_u=\bigcup_{n\geq 0}E_1^{n}[\V{0},u_0]\text{.}
\end{equation}

From now on we will only consider elements of $\mathcal{F}^*$ that are of the form $\sum_k  [\V{x}_k,a_k]^*$ (all the coefficients will be $1$). Thus we shall
consider $E_1^*$ as a transformation acting directly on subsets of $\Psi(Z)\times\mathcal{A}$.

\subsection{$E_1^*$-subtiles and the stepped surface}

\begin{definition}\label{subtiles2}
Let $\pi_{\mathbf{u}}: \R^n\to \mathbb{H}$ be the projection onto $\mathbb{H}$ along the right eigenvector $\mathbf{u}_\alpha$ of $M_\sigma$ corresponding to $\alpha$, renormalized such that $\langle \V{u}_\alpha , \V{v}_\alpha\rangle=1$. Let us consider $\pi_\mathbf{u}\times\text{id}: Z_\text{ext}\to \mathbb{H}_\text{ext}$, where $\text{id}$ is the identity map on $\prod_{\p\mid(\alpha)} K_\p$.
The $E_1^*$-\emph{subtiles} associated with an irreducible Pisot substitution $\sigma$ are defined as
\begin{equation}\label{sub2}
\mathcal{T}_\sigma(a)=\mathop{\rm lim_H}_{k\to\infty} (\pi_{\mathbf{u}}\times\text{id})(M_\sigma^k\cdot (E_1^*)^k[\V{0},a]^*)
\qquad (a\in\mathcal{A})\text{.}
\end{equation}
Furthermore the \emph{central tile} of the substitution is defined as
\begin{equation} \mathcal{T}_\sigma=\bigcup_{a\in\mathcal{A}}\mathcal{T}_\sigma(a)\text{.}
\end{equation}
\end{definition}

Observe that the limit in Equation (\ref{sub2}) exists by a similar argument as the one in \cite[Lemma 3.4]{BSSST:10}.

We now generalize the notion of \emph{stepped surface} given in \cite{AI:00} for the unit case. Define it as the set
\begin{equation}\label{step}  \mathcal{G}_\text{ext}=\{[\Psi(\V{x}),a]^*\in Z_\text{ext}\times\mathcal{A} : \V{x}\in \mathbb{H}^\geq,\;\V{x}-\V{e}_a\in \mathbb{H}^<\}\text{.} \end{equation}
We will provide in Section \ref{7} connections between the two different notions of tiles and stepped surfaces seen so far.
\end{section}

\begin{section}{Model sets}\label{sec:MS}
In this section we describe Sing's approach (see \cite{Sin:06}) of studying Rauzy fractals from the view point of cut-and-project schemes and model sets. We also want to justify here the terminology of calling the set $\Gamma$ defined in \eqref{GammaDef} a {\it translation set}.

\begin{definition}[{see {\it e.g.}~\cite[Section~5]{BM:04}}]
A \emph{cut and project scheme}, or \emph{CPS}, is a triple $(G,H,\tilde{L})$
consisting of a locally compact group $G$  which is the union of countably many compact sets, called the \emph{physical space}, a locally compact group $H$ called the \emph{internal space} and a lattice $\tilde{L}$ in $G\times H$, such that two natural projections $\pi_1:G\times H\rightarrow G$, $\pi_2:G\times H\rightarrow H$ satisfy the following properties:
\begin{enumerate}[(i)]
\item The restriction $\pi_1|_{\tilde{L}}$ is injective.
\item The image $\pi_2(\tilde{L})$ is dense in $H$.
\end{enumerate}
Setting $L=\pi_1(\tilde{L})$, the \emph{star-map} is defined as $(\cdot)^\star=\pi_2\circ(\pi_1|_{\tilde{L}})^{-1}:L\rightarrow H$, and is well-defined on $L$ by injectivity of $\pi_1|_{\tilde{L}}$. With these definitions, we have $\tilde{L}=\{(x,x^\star) : x\in L\}$. We say that a cut and project scheme $(G,H,\tilde{L})$ is \emph{symmetric} if $(H,G,\tilde{L})$ is a cut and project scheme as well.
Given a cut and project scheme $(G,H,\tilde{L})$ and a subset $W\subset H$ define $\Lambda(W)=\{x\in L: x^\star\in W\}$. We call such a set $\Lambda(W)$, or more generally any translate of such a set, a \emph{model set} if $W$ is a non-empty compact set and $W=\overline{\text{int}(W)}$. We say that a model set is \emph{regular} if $\partial W$ has zero Haar measure. In addition, we say that a set $Q$ is an {\it inter model set} if $\Lambda(\text{int}(W))\subset Q \subset \Lambda(W)$.
\end{definition}

A finite family $\underline{\varLambda}=(\varLambda_1,\ldots,\varLambda_n)$ is a \emph{multi-component Delone set} if $\text{supp}(\underline{\varLambda})=\bigcup_{a=1}^n\varLambda_a$ is a Delone set. Similarly we say that $\underline{\varLambda}$ is a \emph{multi-component model set} if each $\varLambda_a=\Lambda(W_a)$ is a model set with respect to the same CPS. The next results can be found in \cite{Sin:06}.

\begin{lemma}\label{le:bef}
$\Phi(V\cdot\Z[\alpha^{-1}])=\alpha\Phi(V\cdot\Z[\alpha^{-1}])$.
\end{lemma}
\begin{proof}
We know that $\alpha V\cdot\Z[\alpha^{-1}]\subset V\cdot\Z[\alpha^{-1}]$, therefore $\alpha\Phi(V\cdot\Z[\alpha^{-1}])\subset \Phi(V\cdot\Z[\alpha^{-1}])$. The set $\alpha\Phi(V\cdot\Z[\alpha^{-1}])$ is a sublattice of $\Phi(V\cdot\Z[\alpha^{-1}])$. Let $D'$ be a fundamental domain of $K_\alpha / \alpha\Phi(V\cdot\Z[\alpha^{-1}])$ and recall that $D$ is a fundamental domain of $K_\alpha / \Phi(V\cdot\Z[\alpha^{-1}])$ (see Lemma \ref{fd}). Then, by (\ref{serre}) and (\ref{product})
\[ \mu_{K_\alpha}(D')=\prod_{\p\in S_\alpha}\abs{\alpha}_\p \cdot \mu_{K_\alpha}(D) = \mu_{K_\alpha}(D)\text{,} \] and the claim follows.
\end{proof}

\begin{proposition}
$(\R,K_\sigma,\Phi(V\cdot\Z[\alpha^{-1}]))$ forms a symmetric cut and project scheme:
\begin{align*}
&\mathbb{R}&\stackrel{\pi_{\p_1}}{\longleftarrow} &\qquad\qquad\quad K_\alpha & \stackrel{\pi}{\longrightarrow} & \quad K_\sigma=
\prod_{\mathfrak{p}\in S_\alpha\setminus\{\p_1\}} K_\mathfrak{p}  \\
&\cup & & \qquad\qquad\quad\cup & & \qquad\qquad\cup   \\
V\cdot \Z& [\alpha^{-1}]&\stackrel{1-1}{\longleftrightarrow} &\qquad\Phi(V\cdot\Z[\alpha^{-1}])
&\stackrel{1-1}{\longleftrightarrow} & \quad\Phi'(V\cdot\Z[\alpha^{-1}])
\end{align*}
\end{proposition}
\begin{proof}
The set $\Phi(V\cdot\Z[\alpha^{-1}])$ is a lattice by Lemma \ref{delalpha}. The projections $\pi_{\p_1}$ and $\pi$ are injective on $\Phi(V\cdot\Z[\alpha^{-1}])$ by construction. By Kronecker's theorem $V$ is dense in $\R$ and so is $V\cdot\Z[\alpha^{-1}]$. It remains to prove that $\Phi'(V\cdot\Z[\alpha^{-1}])$ is dense in $K_\sigma$ (see \cite[Lemma~6.55]{Sin:06}). Since $\Phi(V\cdot\Z[\alpha^{-1}])$ is a lattice in $K_\alpha$, it is relatively dense. Hence $\Phi'(V\cdot\Z[\alpha^{-1}])$ must be relatively dense in $K_\sigma$, \emph{i.e.}, there exists a radius $R> 0$ such that $B(0,R)+\Phi'(V\cdot\Z[\alpha^{-1}])=K_\sigma$. Multiplying this equation by $\alpha$ (which is equivalent to a contraction in $K_\sigma$) and by Lemma \ref{le:bef} we get the denseness.
\end{proof}

For $(Y,d)$ metric space, let $\mathcal{H}(Y)$ be the space of non-empty compact subsets of $Y$, equipped with the Hausdorff metric $d_\mathcal{H}$. In the model set setting, Sing \cite{Sin:06} associates to each primitive substitution $\sigma$ an \emph{expanding matrix function system} $\boldsymbol{\varTheta}$ on $\R^n$ defined by
\begin{equation}
\varTheta_{ab} = \bigcup_{b\stackrel{p}{\longrightarrow}a} \{t_{\delta(p)}\circ f_0 \}\text{, }\quad \text{for }
a,b\in\mathcal{A}\text{,}
\end{equation}
where $f_0(x)=\alpha x$ and $t_{\delta(p)}(x)=x+\delta(p)$. Then its substitution matrix $\boldsymbol{S}\boldsymbol{\varTheta}=(\abs{\varTheta_{ab}})_{a,b\in\mathcal{A}}$ equals $M_\sigma$.
Given such a $\boldsymbol{\varTheta}$ we can define the \emph{adjoint iterated function system} $\boldsymbol{\varTheta}^\#$ on $\mathcal{H}(\R)^n$ by
\begin{equation}
\varTheta^\#_{ab} = \bigcup_{a\stackrel{p}{\longrightarrow}b} \{f^{-1}_0\circ t_{\delta(p)} \}
\text{, }\quad \text{for }a,b\in\mathcal{A}\text{.}
\end{equation}
Then obviously $\boldsymbol{S}\boldsymbol{\varTheta}^\#=M_\sigma^t$. Note that $\boldsymbol{\varTheta}^\#$ is just a way to write a {\it graph directed iterated function system} in the sense of Mauldin and Williams~\cite{MW:88} in matrix form. By the general theory of graph directed iterated function systems (see \cite{MW:88}) there exists a unique attractor for $\boldsymbol{\varTheta}^\#$, and it is easy to see that it is $\underline{A}=(A_a)_{a\in\mathcal{A}}\subset\mathcal{H}(\R)^n$, where the $A_a=[0,\delta(a)]$ are called \emph{natural intervals}.

Geometrically we can interpret $\sigma$ as a \emph{tiling of the line}: given a fixed point $u=u_0u_1\cdots\in\mathcal{A}^\omega$ of $\sigma$, we represent each letter $a$ by the ``type $a$'' interval $A_a$; starting with the first of these intervals we can construct the entire line inflating repetitively $A_a$ by $\alpha$ and subdividing it into the corresponding intervals given by the substitution (compare this to the action of the one-dimensional geometric realization $E_1$ defined in Section~\ref{subsec51}; in particular, we refer to \eqref{LuStar}).

Given the tiling of the line, denote the set of left endpoints of the type $a$ intervals by $\varLambda_a$. Precisely, define $\underline{\varLambda}=(\varLambda_a)_{a\in\mathcal{A}}$ by
\begin{equation}\label{lambda} \underline{\varLambda}=\bigcup_{k\geq 0} \boldsymbol{\varTheta}^k(\emptyset,\cdots,\emptyset,\{0\},
\emptyset,\cdots,\emptyset)^t\text{,} \end{equation} where $\{0\}$ is at position $u_0$.
Then $\underline{\varLambda}=(\varLambda_a)_{a\in\mathcal{A}}$ is a \emph{substitution multi-component Delone set}, \emph{i.e.}, $\underline{\varLambda}=\boldsymbol{\varTheta}(\underline{\varLambda})$ and together with $\underline{A}=\boldsymbol{\varTheta^\#}(\underline{A})$ this forms the \emph{representation with natural intervals} $\underline{\varLambda}+\underline{A}$ of a fixed point $u$ of $\sigma$.\footnotemark[2]

\begin{ex}
Consider the substitution of Example \ref{ex1}, $\sigma(1)=121$, $\sigma(2)=11$. Then we obtain the following expanding matrix function system $\boldsymbol{\varTheta}$ and its adjoint iterated function system $\boldsymbol{\varTheta}^\#$:
\[  \boldsymbol{\varTheta}=\begin{pmatrix} \{f_0,f_{\delta(12)}\} & \{f_0,f_{\delta(1)}\} \\ \{f_{\delta(1)}\} & \emptyset \end{pmatrix}\text{, }\qquad
\boldsymbol{\varTheta}^\#=\begin{pmatrix} \{g_0,g_{\delta(12)}\} & \{g_{\delta(1)}\} \\ \{g_0,g_{\delta(1)}\} & \emptyset \end{pmatrix}\text{,}  \] where $f_d(x)=\alpha x+ d$ and $g_d(x)=\alpha^{-1}(x+d)$, for $d\in\mathcal{D}=\{\delta(\epsilon),\delta(1),\delta(12)\}$.
We get the tiling of the line applying repetitively the process of inflation and subdivision on the interval $[0,\delta(1)]$: 
\[
\begin{tikzpicture}[scale=0.6]
\draw (0,-.1)--(0,.1)node[above]{$0$} (2.732/2,-.1)--(2.732/2,.1)node[above]{$\frac{\alpha}{2}$};
\draw[very thick,blue]  (0,0)--(2.732/2,0);
\end{tikzpicture} \;\stackrel{\cdot\alpha}{\longmapsto}\;
\begin{tikzpicture}[scale=0.6]
\draw (0,-.1)--(0,.1)node[above]{$0$} (2.732/2,-.1)--(2.732/2,.1) (2.732/2+1,-.1)--(2.732/2+1,.1) (2.732+1,-.1)--(2.732+1,.1)node[above]{$\alpha+1$};
\draw[very thick,blue]  (0,0)--(2.732/2,0) (2.732/2+1,0)--(2.732+1,0);
\draw[thick,green] (2.732/2,0)--(2.732/2+1,0);
\end{tikzpicture} \;\stackrel{\cdot\alpha}{\longmapsto} \;
\begin{tikzpicture}[scale=0.6]
\draw (0,-.1)--(0,.1)node[above]{$0$} (2.732/2,-.1)--(2.732/2,.1) (2.732/2+1,-.1)--(2.732/2+1,.1) (2.732+1,-.1)--(2.732+1,.1)
(3.732+2.732/2,-.1)--(3.732+2.732/2,.1) (3.732+2.732,-.1)--(3.732+2.732,.1)
(3.732+3*2.732/2,-.1)--(3.732+3*2.732/2,.1) (4.732+3*2.732/2,-.1)--(4.732+3*2.732/2,.1) (4.732+2*2.732,-.1)--(4.732+2*2.732,.1)node[above]{$3\alpha+2$};
\draw[very thick,blue]  (0,0)--(2.732/2,0) (2.732/2+1,0)--(2.732+1,0) (2.732+1,0)--(3.732+2.732/2,0)
(3.732+2.732/2,0)--(3.732+2.732,0) (3.732+2.732,0)--(3.732+3*2.732/2,0) (4.732+3*2.732/2,0)--(4.732+2*2.732,0);
\draw[thick,green] (2.732/2,0)--(2.732/2+1,0) (3.732+3*2.732/2,0)--(4.732+3*2.732/2,0);
\end{tikzpicture} \;\stackrel{\cdot\alpha}{\longmapsto}\; \cdots
 \]
Furthermore we have that the sets $\varLambda_a$ of left endpoints of the type $a$ intervals, for $a\in\mathcal{A}$, satisfy the point set equations
\begin{align*}
\varLambda_1 &= \alpha\varLambda_1 \cup  \alpha\varLambda_1 + \delta(12) \cup \alpha\varLambda_2 \cup \alpha\varLambda_2 + \delta(1)\text{,}   \\
\varLambda_2 &= \alpha\varLambda_1 + \delta(1)\text{,}
\end{align*}
and the natural intervals satisfy
\begin{align*}
A_1 &= \alpha^{-1}A_1 \cup  \alpha^{-1}(A_1 + \delta(12)) \cup \alpha^{-1}(A_2+\delta(1))\text{,}   \\
A_2 &= \alpha^{-1}A_1 \cup \alpha^{-1}(A_1+\delta(1))\text{.}
\end{align*}
\end{ex}
\footnotetext[2]{We could have defined $\boldsymbol{\varTheta}$ using the functions $t_{-\delta(p)}\circ f_0$, to obtain by duality $\boldsymbol{\varTheta}^\#$. 
In this way, we would have obtained a negative tiling of the line $-(\underline{\varLambda}+\underline{A})$, with $\varLambda_a$ set of right endpoints of the type $a$ intervals.}

We can extend $\boldsymbol{\varTheta}$ to the graph directed iterated function system $\boldsymbol{\varTheta}^\star$ on $\mathcal{H}(K_\sigma)^n$
 relative to the CPS $(\R,K_\sigma,\Phi(V\cdot\Z[\alpha^{-1}]))$ with star-map $\Phi'$. 
As done before we can consider the adjoint $(\boldsymbol{\varTheta}^\star)^\#$ on $K_\sigma^n$ relative to the CPS $(\R,K_\sigma,\Phi(V\cdot\Z[\alpha^{-1}]))$, 
which is an expanding matrix function system. This can now be used to define Rauzy fractals in this context.

\begin{definition}
Let $\underline{\varOmega}=(\varOmega_a)_{a\in\mathcal{A}}\subset\mathcal{H}(K_\sigma)^n$ be the solution of the graph directed iterated function system $\boldsymbol{\varTheta}^\star(\underline{\varOmega})=\underline{\varOmega}$. $\underline{\varOmega}$ is called the \emph{dual prototile}.
The regular multi-component inter model set $\underline{\varUpsilon}=(\varUpsilon_a)_{a\in\mathcal{A}}$ in $K_\sigma^n$ associated with the CPS $(K_\sigma,\R,\Phi(V\cdot\Z[\alpha^{-1}]))$, defined by $\varUpsilon_a=\Lambda([0,\delta(a)))$, is called \emph{translation set}.
\end{definition}

Observe that, for $a\in\mathcal{A}$,
\begin{equation}\label{tran}
\varUpsilon_a=\Lambda([0,\delta(a)))=\{\pi(\V{z})\in K_\sigma: \V{z}=(z_\p)_\p\in \Phi(V\cdot\Z[\alpha^{-1}]), z_{\p_1}\in [0,\delta(a)) \}=\Phi'(\text{Frac}(\sigma,a))\text{,}
\end{equation}
which shows that the translation set $\Gamma=\text{supp}(\underline{\varUpsilon})$.

\begin{lemma}\label{del}
$\Gamma$ is a Delone set.
\end{lemma}

\begin{proof}
It suffices to prove that each $\Phi'(\text{Frac}(\sigma,a))$ is a model set, because by \cite[Proposition 2.6]{MO:97} model sets are Delone sets. 
But this follows from (\ref{tran}).
\end{proof}

\end{section}

\begin{section}{Relations between different approaches}\label{7}

The aim of this section is to provide connections between the different approaches seen so far. We start establishing relations between the geometric realization of a 
substitution and Dumont-Thomas numeration.

\begin{lemma}
The $\Q$-basis $\{v_1,\ldots,v_n\}$ of $\Q(\alpha)$ has the property that multiplication by $\alpha$ in $\Q(\alpha)$ is conjugate to the action of $M_\sigma$ on $\Q^n$, that is, the diagram
\[
\xymatrix{
\Q(\alpha) \ar[r]^{\cdot\alpha} \ar[d]_f & \Q(\alpha) \ar[d]^f\\
\Q^n  \ar[r]^{M_\sigma} & \Q^n  }
\]
commutes. Here $f:\Q(\alpha)\rightarrow\Q^n$
maps $z=z_1v_1+\cdots+z_nv_n$ to the vector $(z_1,\ldots,z_n)^t$.
Furthermore $V\cdot\Z[\alpha^{-1}]\cong Z$ as $\Z$-modules.
\end{lemma}

\begin{proof}
The commutativity of the diagram follows because $\mathbf{v}_\alpha$ is a left eigenvector of $M_\sigma$.
The isomorphism $V\cdot\Z[\alpha^{-1}]\cong Z$ follows because $V=\langle v_1,\ldots,v_n\rangle_\Z\cong \Z^n$ and the multiplication by $\alpha$ on $V$ is conjugate to the action of $M_\sigma$ on $Z$.
\end{proof}

Regarding $\mathbb{Q}(\alpha)$ as a vector space with basis $\{v_1,\ldots,v_n\}$ is particularly convenient because, for every $\V{x}\in Z$, the scalar product $\langle\V{x},\V{v}_\alpha\rangle$ is an element of $\mathbb{Q}(\alpha)$.
We use the connection between $V\cdot\Z[\alpha^{-1}]$ and $Z$ to associate with every $\Psi(\V{x})\in Z_\text{ext}$ the embedding of $\langle\V{x},\V{v}_\alpha\rangle\in\mathbb{Q}(\alpha)$ in $K_\alpha$. Precisely we have the following statement.

\begin{lemma}\label{comm}
The following diagram is commutative:
\[ \xymatrix{
Z\ar[r]^\Psi \ar[d]_{f^{-1}} & Z_{\rm ext} \ar[d]_F \ar[r]^{\pi_\mathbf{u}\times\rm id} & \mathbb{H}_{\rm ext}
\ar[d]^G \\ V\cdot\Z[\alpha^{-1}] \ar[r]^\Phi & K_\alpha \ar[r]^{\pi} & K_\sigma } \]
where
\begin{equation}\label{F}
F=\begin{pmatrix}
\Phi_\infty(v_1)^t & \cdots & \Phi_\infty(v_n)^t &  & \mathbf{0} &  \\
 & \mathbf{0} &  &  & I &
\end{pmatrix}\text{,}
\end{equation}
$I$ denotes the identity matrix whose size equals the number of finite primes $\p$ satisfying $\p \mid(\alpha)$, and $G$ is the matrix obtained from $F$ by erasing the first row.
\end{lemma}

\begin{proof}
For the left-square diagram, let $\V{x}\in Z$, then $\Psi(\V{x})=(\V{x},\langle \V{x},\V{v}_\alpha \rangle,\ldots,\langle \V{x},\V{v}_\alpha \rangle)$ and applying $F$ 
we get the vector $(\Phi_\infty(\langle \V{x},\V{v}_{\alpha}\rangle),\langle \V{x},\V{v}_\alpha \rangle,\ldots,\langle \V{x},\V{v}_\alpha \rangle)$.
Now, embedding $f^{-1}(\V{x})=\langle \V{x},\V{v}_\alpha\rangle$ by $\Phi$ we get exactly the same vector.

For the right-square diagram, let $\V{x}=(\V{x}_1,\V{x}_2)\in Z_\text{ext}$, where $\V{x}_1\in Z$ and $\V{x}_2\in\prod_{\p\mid(\alpha)}K_\p$. 
Let $\alpha^{(i)}$, $i=2,\ldots,n$, be the Galois conjugates of $\alpha=\alpha^{(1)}$. 
Let $\mathbf{v}_{\alpha^{(i)}}$ and $\mathbf{u}_{\alpha^{(i)}}$ be respectively the left and right eigenvectors for $\alpha^{(i)}$. 
Taking $\mathbf{u}_{\alpha}$ such that $\langle\mathbf{u}_{\alpha},\mathbf{v}_{\alpha}\rangle=1$, it follows easily that $\{\mathbf{u}_{\alpha^{(i)}}\}_{i=1}^n$, 
$\{\mathbf{v}_{\alpha^{(i)}}\}_{i=1}^n$ are dual bases. Then every $\mathbf{x}_1\in Z\subset\R^n$ admits the decomposition (cf. \cite[Section 2.1]{CS:01})
\[ \mathbf{x}_1 = \langle\mathbf{x}_1,\mathbf{v}_\alpha\rangle \mathbf{u}_\alpha + \sum_{i=2}^n \langle\mathbf{x}_1,\mathbf{v}_{\alpha^{(i)}}\rangle \mathbf{u}_{\alpha^{(i)}}\text{.} \]
Therefore
\[ 
(\pi_\mathbf{u}\times\text{id})(\mathbf{x})=(\langle \mathbf{x}_1,\mathbf{v}_{\alpha^{(2)}} \rangle \mathbf{u}_{\alpha^{(2)}}+\ldots+
\langle \mathbf{x}_1,\mathbf{v}_{\alpha^{(n)}}\rangle \mathbf{u}_{\alpha^{(n)}},\V{x}_2) \text{.}
\]
Applying $G$ to this vector we get
\[ G(\pi_\mathbf{u}\times\text{id})(\mathbf{x})=(\langle \mathbf{x}_1,\mathbf{v}_{\alpha^{(2)}} \rangle,\ldots,
\langle \mathbf{x}_1,\mathbf{v}_{\alpha^{(n)}}\rangle,\V{x}_2)\text{,}  \]
since $\{\mathbf{u}_{\alpha^{(i)}}\}_{i=1}^n$ and $\{\mathbf{v}_{\alpha^{(i)}}\}_{i=1}^n$ are dual bases.
We obtain the same result projecting $F\V{x}=(\Phi_\infty(\langle \V{x}_1,\V{v}_{\alpha}\rangle),\V{x}_2)$ by $\pi$.
\end{proof}

Notice that $F$ is invertible because $\{v_1,\ldots,v_n\}$ are rationally independent.
The following result surveys and clarifies the connection established between the Dumont-Thomas numeration and geometric representations of substitutions.

\begin{theorem}\label{diagramtheorem}
Each square in the diagram
\[
\xymatrix@!0{
& & Z \ar[ddll]_{\Psi}\ar[rrrrr]^{M_\sigma} \ar@{.>}'[dd][dddd]        
& & & & & Z \ar[ddll]_{\Psi} \ar@{.>}[dddd]^{f^{-1}}
\\ & & & & & & & \\
Z_{\text{\rm ext}} \ar[rrrrr]^{\qquad M_\sigma}\ar@{.>}[dddd]_{F}
& & & & & Z_{\text{\rm ext}}\ar@{.>}[dddd] & & \\
&&&&&&&\\
& & V\cdot\Z[\alpha^{-1}]  \ar'[rrr]^(.7){\cdot\alpha}[rrrrr]\ar[ddll]^\Phi
& & & & & V\cdot\Z[\alpha^{-1}] \ar[ddll]^\Phi
\\ & & & & & & & \\
K_\alpha \ar[rrrrr]^{\cdot\alpha}
& & & & & K_\alpha & &
}
\]
is commutative.
\end{theorem}

Theorem \ref{diagramtheorem} provides the connection between the two stepped surfaces $\mathcal{G}_\text{ext}$ in (\ref{step}) and $\mathcal{S}$ in (\ref{stepped}). 
Indeed, considering a point $\Psi(\V{x})\in Z_\text{ext}$ is equivalent to taking a point $\Phi(\langle\V{x},\V{v}_\alpha\rangle)\in \Phi(V\cdot\Z[\alpha^{-1}])$, 
up to conjugation. Furthermore, the condition $\V{x}\in \mathbb{H}^\geq$, $\V{x}-\V{e}_a\in \mathbb{H}^<$ can be translated in the $\mathbb{Q}(\alpha)$-world to
$0\leq\langle \V{x},\V{v}_\alpha\rangle<\langle \V{e}_a,\V{v}_\alpha\rangle=\delta(a)$.

The advantage of dealing in the number field world, \emph{i.e.}, with $K_\alpha$ instead of the space $Z_{\text{ext}}$, is that the stepped surface and the translation 
set are parametrized by one coordinate $x\in V\cdot\Z[\alpha^{-1}]\cap[0,\delta(a))$, for some $a\in\mathcal{A}$.

By Theorem \ref{diagramtheorem}, $\T$ and the dual geometrical substitution $E_1^*$ are conjugate.

\begin{lemma}\label{conj}
The following diagram is commutative:
\[
\xymatrix{  \Psi(Z)\times\mathcal{A}\; \ar[r]^{E_1^*}\ar[d]_F & \; 2^{\Psi(Z)\times\mathcal{A}} \ar[d]^F  \\
K_\alpha\times\mathcal{A}\; \ar[r]^\T &  \; 2^{K_\alpha\times\mathcal{A}}  }
\]
with $F$ defined in Equation (\ref{F}) and where by convention $F[\mathbf{x},a]^*$ equals $(F(\mathbf{x}),a)$.
\end{lemma}

Next we provide the relations between the Dumont-Thomas numeration and the model set approach, and we connect the latter with the maps $E_1$ and $E_1^*$.

We can refine the sets $\Z_{\sigma,a}^{(k)}$ defined in Section~\ref{fracint} by taking only those finite integers associated with walks in the prefix automaton starting 
at a state $b$ and ending at state $a$. Call these sets $\Z_{b,a}^{(k)}$. By definition $\Z_{\sigma,a}^{(k)}=\bigcup_{b\in\mathcal{A}}\Z_{b,a}^{(k)}$. 
Moreover, if the word $\sigma(b)$ starts with $b$ we easily see that the sequence $(\Z_{b,a}^{(k)})_{k\ge 0}$ is nested.

\begin{lemma}\label{u0}
Let $u=u_0u_1\cdots$ be the fixed point of $\sigma$ and let $\underline{\varLambda}$ be as in (\ref{lambda}).
Then we have $(\Z_{u_0,a})_{a\in\mathcal{A}}=\underline{\varLambda}$. 
Furthermore $(\Z_{\sigma,a})_{a\in\mathcal{A}}=\mathop{\rm Lim}_{k\to\infty}\boldsymbol{\varTheta}^k(\{0\})_{a\in\mathcal{A}}$, 
and in particular $(\Z_{\sigma,a})_{a\in\mathcal{A}}=\boldsymbol{\varTheta}(\Z_{\sigma,a})_{a\in\mathcal{A}}$.
\end{lemma}

\begin{proof}
We have $\boldsymbol{\varTheta}^k(\emptyset,\dots,\emptyset,\{0\},\emptyset,\dots,\emptyset)^t=
(\boldsymbol{\varTheta}^k_{a,u_0}(0))_{a\in\mathcal{A}}$, whose elements are of the form $\alpha^k\delta(p_k)+\cdots+\delta(p_0)$ where 
$u_0\stackrel{p_k}{\rightarrow}\cdots\stackrel{p_0}{\rightarrow}a$. But these are elements of $\Z_{u_0,a}^{(k)}$. 
Recalling that $\Z_{\sigma,a}^{(k)}=\bigcup_{u_0\in\mathcal{A}}\Z_{u_0,a}^{(k)}$, we get the second statement.
\end{proof}

\begin{lemma}\label{another}
Let $h: \Psi(Z)\times\mathcal{A}  \to \R^n$, $[\Psi(\mathbf{y}),a]\mapsto(\emptyset,\ldots,\emptyset,\{\langle -\mathbf{y},\mathbf{v}_\alpha\rangle\},\emptyset,\ldots,\emptyset)^t$,
$v: X \to \R^n$, $(x,a)\mapsto (\emptyset,\ldots,\emptyset,\{x\},\emptyset,\ldots,\emptyset)^t$, with $x$ at position $a$, and let 
$C=v\circ G\circ (\pi_\textbf{u}\times \text{id})$, where $G$ is as in Lemma \ref{comm} and $\pi_\mathbf{u}\times \text{id}$ as in Definition \ref{subtiles2}. 
Then the diagrams
\[  \xymatrix{
\Psi(Z)\times\mathcal{A} \ar[r]^{E_1} \ar[d]_h  & 2^{\Psi(Z)\times\mathcal{A}} \ar[d]^h \\
\R^n \ar[r]^{\boldsymbol{\varTheta}}  &  \R^n  }\qquad
\xymatrix{
X  \ar[r]^{T_\sigma^{-1}}\ar[d]_{v}  & 2^X \ar[d]^{v} \\
\R^n \ar[r]^{\boldsymbol{\varTheta^\#}}  &  \R^n    }\qquad
\xymatrix{
\Psi(Z)\times\mathcal{A} \ar[r]^{E_1^*} \ar[d]_{C}  & 2^{\Psi(Z)\times\mathcal{A}} \ar[d]^{C} \\
K_\sigma^n \ar[r]^{\boldsymbol{(\varTheta^\star)^\#}}  &  K_\sigma^n  }  \]
are commutative.
\end{lemma}

\begin{proof}
For the first diagram we have
\begin{align*}
\boldsymbol{\varTheta}(h([\Psi(\V{x}),a])) &= (\varTheta_{1a}(\langle -\V{x},\V{v}_\alpha\rangle),\ldots,\varTheta_{na}(\langle -\V{x},\V{v}_\alpha\rangle) )^t \\
&=\Big(\bigcup_{a\stackrel{p}{\longrightarrow}1}\{\alpha \langle -\V{x},\V{v}_\alpha\rangle+\delta(p)\},\cdots,\bigcup_{a\stackrel{p}{\longrightarrow}n}\{\alpha \langle -\V{x},\V{v}_\alpha\rangle+\delta(p)\}\Big)^t \\
&=h\Big(\sum_{a\stackrel{p}{\longrightarrow}b}[\Psi(M_\sigma\V{x}-\V{P}(p)),b]\Big)=h(E_1[\Psi(\V{x}),a])\text{.}
\end{align*}
For the second diagram 
\begin{align*}
\boldsymbol{\varTheta}^\#(v(x,a)) &= (\varTheta^\#_{1a}(x),\cdots,\varTheta^\#_{na}(x) )^t
= \Big(\bigcup_{1\stackrel{p}{\longrightarrow}a} \{\alpha^{-1}(x+\delta(p))\},\cdots,\bigcup_{n\stackrel{p}{\longrightarrow}a} \{\alpha^{-1}(x+\delta(p))\} \Big)^t \\
&= v\Big( \bigcup_{b\stackrel{p}{\longrightarrow}a} \{(\alpha^{-1}(x+\delta(p)),b)\} \Big)
= v(T_\sigma^{-1}(x,a))\text{.}
\end{align*}
By the second diagram $(\boldsymbol{\varTheta}^\star)^\#\circ (v\circ \pi)=(v\circ \pi)\circ \T$, if we denote again by $v$ the map $v$ acting on $K_\sigma\times\mathcal{A}$. Moreover we know that $\T$ and $E_1^*$ are conjugate by $F$. Then, using $\pi\circ F=G\circ (\pi_\textbf{u}\times \text{id})$, the commutativity of the third diagram follows by $C$.
\end{proof}

It has been already shown in Section \ref{sec:MS} that the two translation sets $\underline{\varUpsilon}$ and $\Gamma$ are the same.
By Lemma \ref{another} we can translate Proposition \ref{inv} to $\underline{\varUpsilon}=(\boldsymbol{\varTheta}^\star)^\#(\underline{\varUpsilon})$
({\it cf.}~\cite[Proposition 6.72]{Sin:06}).

We end this section by giving the relation between all kinds of Rauzy fractals defined so far.

\begin{theorem}
Let $\sigma$ be an irreducible Pisot substitution over the alphabet $\mathcal{A}$ and fix $a\in\mathcal{A}$. Let $\mathcal{R}_\sigma(a)$ be the associated Dumont-Thomas subtile, $\mathcal{T}_\sigma(a)$ the associated $E_1^*$-subtile and
$\varOmega_a$ the associated dual prototile. If $G$ is the matrix defined in Lemma \ref{comm}, then
\[ 
\mathcal{R}_\sigma(a)=\varOmega_a=G\;\mathcal{T}_\sigma(a)\text{.} 
\] 
\end{theorem}

\begin{proof}
We start with proving the first identity. Recall that
$\underline{\varOmega}=(\varOmega_a)_{a\in\mathcal{A}}$ is the attractor of the graph directed iterated function system $\boldsymbol{\varTheta}^\star$, 
and the Dumont-Thomas subtiles are defined by  $\mathcal{R}_\sigma(a)=\overline{\Phi'(\Z_{\sigma,a})}$. Then by Lemma \ref{u0}
\[ (\mathcal{R}_\sigma(a))_{a\in\mathcal{A}}=
\boldsymbol{\varTheta}^\star(\overline{\Phi'(\Z_{\sigma,a})})_{a\in\mathcal{A}}=
\boldsymbol{\varTheta}^\star(\mathcal{R}_\sigma(a))_{a\in\mathcal{A}} \] and the result follows by uniqueness of the attractor of $\boldsymbol{\varTheta}^\star$.

To prove the second identity, starting from \eqref{subnew} we get
\begin{align*}
\mathcal{R}_\sigma(a)&= \mathop{\rm lim_H}_{k\to\infty} \pi(\alpha^k T_{\rm ext}^{-k}(\V{0},a)) \\
&=  \mathop{\rm lim_H}_{k\to\infty} \pi(\alpha^k T_{\rm ext}^{-k}(F(\V{0}),a)) \\
&=  \mathop{\rm lim_H}_{k\to\infty} \pi(\alpha^k F(E_1^*)^k[\V{0},a]^*) \;\;\quad\qquad\qquad\text{ (by Lemma }\ref{conj}\text{)} \\
&=  \mathop{\rm lim_H}_{k\to\infty} \pi(F M_\sigma^k(E_1^*)^k[\V{0},a]^*) \;\qquad\quad\qquad\text{ (by Theorem }\ref{diagramtheorem}\text{)} \\
&= \mathop{\rm lim_H}_{k\to\infty} G (\pi_\V{u}\times\text{id})(M_\sigma^k(E_1^*)^k[\V{0},a]^*) \qquad\text{ (by Lemma }\ref{comm}\text{)}  \\
&= G\; \mathcal{T}_\sigma(a)\text{.} \qedhere
\end{align*}
\end{proof}

\begin{remark}
Rauzy fractals can be defined alternatively as the closure of the projections onto the contracting hyperplane of vertices of the broken line associated 
to the fixed point $u=u_0u_1\cdots$ of the substitution (see \emph{e.g.} \cite{BS:05,BST:10}). It turns out that this is equivalent to taking the closure of the embedding by $\Phi'$ of
$\bigcup_{k\geq 0} \Z_{u_0,a}^{(k)}$, for each $a\in\mathcal{A}$. In fact, these Rauzy fractals coincide with the Dumont-Thomas central tiles, since they are both solutions of
the graph directed iterated function system $\boldsymbol{\varTheta}^\star$.
\end{remark}

\end{section}

\begin{section}{Basic properties of the tiles}
In all what follows we will work in the number field setting, \emph{i.e.}, with $\T$. For the sake of simplicity we will denote $\T$ acting on $K_\sigma\times\mathcal{A}$ again by $\T$, instead of writing $\pi\circ \T$.
For this reason we write the Dumont-Thomas subtiles as
\begin{equation}\label{ssub}
\mathcal{R}_\sigma(a)=\mathop{\rm lim_H}_{k\to\infty} \alpha^k\cdot T_{\rm ext}^{-k}(\V{0},a)\text{, }\quad\text{for }a\in\mathcal{A}\text{.}
\end{equation}
We show now some of their properties.

\subsection{Topological properties}

The following proposition gives information on the $\p$-adic height of the central tile.
\begin{proposition}\label{pheight}
If $\mathbf{z}=(z_\p)_{\p\in S_\alpha\setminus\{\p_1\}}\in\mathcal{R}_\sigma$, for every $\p\mid(\alpha)$ we have $z_\p\in\mathfrak{p}^{d_\mathfrak{p}}$, where $d_\mathfrak{p}=\min\{v_\mathfrak{p}(x) : x\in V\}$.
\end{proposition}
\begin{proof}
If $\V{z}\in\mathcal{R}_\sigma$, we can write $\V{z}=\sum_{k=0}^\infty\Phi'(\delta(p_k)\alpha^k)$.
For $\p\mid(\alpha)$, the $\p$-th component of $\V{z}$ is $z_\p=\sum_{k=0}^\infty \delta(p_k)\alpha^k\in K_\p$. We deduce that $v_\p(z_\p)\geq \min_{k\geq 0}\{v_\p(\delta(p_k)\alpha^k)\}=\min_{\delta(p)\in\mathcal{D}}\{v_\p(\delta(p))\}=d_\p$, thus $z_\p\in \p^{d_\p}$.
\end{proof}

Next we prove that the subtiles cover the representation space ({\it cf.}~\cite{ABBS:08} for the non-unit beta-expansion setting).
\begin{proposition}\label{coverprop}
Let $\sigma$ be an irreducible Pisot substitution. The subtiles $\mathcal{R}_\sigma(a)$
provide a uniformly locally finite covering of the representation space $K_\sigma$ governed by $\Gamma$, in particular
\[ K_\sigma= \bigcup_{(\gamma,a)\in\Gamma}\mathcal{R}_\sigma(a)+\gamma\text{.} \]
\end{proposition}

\begin{proof}
Let $\mathcal{C}_\sigma=\bigcup_{(\gamma,a)\in\Gamma}\mathcal{R}_\sigma(a)+\gamma$. Every point of $\mathcal{C}_\sigma$ is of the form $\mathbf{z}+\Phi'(x)$, where $\mathbf{z}=\sum_{i\geq 0}\Phi'(\delta(p_i)\alpha^i)\in\mathcal{R}_\sigma(a)$, $x=\sum_{i\geq 1}\delta(p_{-i})\alpha^{-i}\in\text{Frac}(\sigma,a)$. We have $\alpha \mathcal{C}_\sigma\subseteq \mathcal{C}_\sigma$, since by Lemma~\ref{le:T} we have $T_\sigma(x,a)\in\text{Frac}(\sigma,b)$ and $\alpha \mathbf{z}+\Phi'(\delta(p_{-1}))\in\mathcal{R}_\sigma(b)$, for some $b\in\mathcal{A}$. By Lemma \ref{del}, $\mathcal{C}_\sigma$ is relatively dense in $K_\sigma$. Furthermore, as $\alpha \mathcal{C}_\sigma\subseteq \mathcal{C}_\sigma$ and $\alpha$ is a contraction, $\mathcal{C}_\sigma$ is dense in $K_\sigma$. By compactness of the subtiles and uniformly discreteness of $\Gamma$ we obtain $K_\sigma=\mathcal{C}_\sigma$.
\end{proof}

In the following theorem we state important properties of our tiles ({\it cf.}~\cite[Corollary 6.66]{Sin:06}).

\begin{theorem}\label{basictheorem}
The following assertions hold for the subtiles $\mathcal{R}_\sigma(a)$, $a \in \mathcal{A}$, of an irreducible Pisot substitution.

\begin{enumerate}[(i)]
\item  The subtiles $\mathcal{R}_\sigma(a)$ are the solution of the graph directed iterated function system
\begin{equation}\label{gifs}
\mathcal{R}_\sigma(a)=\bigcup_{(\gamma,b)\in\T(\V{0},a)} \alpha(\mathcal{R}_\sigma(b)+\gamma)=\bigcup_{b\stackrel{p}{\longrightarrow}a}
\alpha\mathcal{R}_\sigma(b)+\Phi'(\delta(p))\text{,}
\end{equation}
where the union is measure disjoint. \label{basicone}

\item Each subtile $\mathcal{R}_\sigma(a)$ is the closure of its interior.  \label{basictwo}

\item The boundary of each subtile $\mathcal{R}_\sigma(a)$ has Haar measure zero.  \label{basicthree}
\end{enumerate}
\end{theorem}

\begin{proof}
\eqref{basicone}
Equation \eqref{gifs} is a direct consequence of Lemma~\ref{u0}, but we prefer to give here an explicit proof.
By \eqref{ssub} and \eqref{trick} we obtain
\begin{align*}
\mathcal{R}_\sigma(a) &= \mathop{\rm lim_H}_{k\to\infty} \alpha^k\cdot T_{\rm ext}^{-k}(\V{0},a) =
\alpha \mathop{\rm lim_H}_{k\to\infty} \bigcup_{(\gamma,b)\in \T(\V{0},a)}\alpha^{k-1}\cdot
T_{\rm ext}^{-(k-1)}(\gamma,b) \\
&= \alpha \bigcup_{(\gamma,b)\in \T(\V{0},a)}(\mathcal{R}_\sigma(b)+\gamma) =
\bigcup_{b\stackrel{p}{\longrightarrow}a}\alpha\mathcal{R}_\sigma(b)+\Phi'(\delta(p))\text{.}
\end{align*}

Let $\V{m}=(\mu(\mathcal{R}_\sigma(a)))_{a\in \mathcal{A}}$. Applying the measure $\mu$ to equation (\ref{gifs}) gives
\begin{equation}\label{disj}
\mu(\mathcal{R}_\sigma(a)) \leq
\sum_{b\stackrel{p}{\longrightarrow} a}
\mu(\alpha\mathcal{R}_\sigma(b)+\Phi'(\delta(p)))
=\alpha^{-1}\sum_{b\stackrel{p}{\longrightarrow} a}\mu(\mathcal{R}_\sigma(b))
=\alpha^{-1}\sum_{b\in \mathcal{A}}(M_\sigma)_{ab}\mu(\mathcal{R}_\sigma(b))\text{.}
\end{equation}
So we showed that the vector $\textbf{m}$ satisfies $M_\sigma\V{m}\geq \alpha \V{m}$, and, as a direct consequence of the Perron-Frobenius Theorem, we get $M_\sigma\V{m}=\alpha\V{m}$. Thus the inequality in (\ref{disj}) is actually an equality, and thus no overlap with positive measure occurs in the union in (\ref{gifs}).

\eqref{basictwo}
Since $K_\sigma$ is locally compact, we deduce from Baire's theorem that there exists $a\in \mathcal{A}$ such that ${\rm int}(\mathcal{R}_\sigma(a))\not=\emptyset$. Therefore (\ref{gifs}) and the primitivity of $\sigma$ yield that ${\rm int}(\mathcal{R}_\sigma(a))\not=\emptyset$ holds for each $a\in \mathcal{A}$. Let now $a\in \mathcal{A}$ and consider $\eta\in \mathcal{R}_\sigma(a)$. Let $B$ be an open ball centered at $\eta$. It suffices to show that $B \cap {\rm int}(\mathcal{R}_\sigma(a))\not=\emptyset$. Using the $k$-fold iteration
\begin{equation}\label{kfold}
\mathcal{R}_\sigma(a)=\bigcup_{(\gamma,b)\in T_{\rm ext}^{-k}(\V{0},a)} \alpha^k(\mathcal{R}_\sigma(b)+\gamma)
\end{equation}
of \eqref{gifs} for $k$ large enough, we obtain that $\alpha^k(\mathcal{R}_\sigma(b)+\gamma)\subseteq B$ holds for some $(\gamma,b)\in T_{\rm ext}^{-k}(\V{0},a)$. As ${\rm int}(\alpha^k(\mathcal{R}_\sigma(b)+ \gamma))\not=\emptyset$ the ball $B$ contains inner points of $\mathcal{R}_\sigma(a)$.

\eqref{basicthree}
Let $B\subset {\rm int}(\mathcal{R}_\sigma(a))$ be an open ball and fix $b\in \mathcal{A}$. By the primitivity of $\sigma$ we may choose $k\in\mathbb{N}$ large enough such that $U:=\alpha^k(\mathcal{R}_\sigma(b) + \gamma)\subseteq B$ holds for some $(\gamma,b)\in T_{\rm ext}^{-k}(\V{0},a)$. The boundary $\partial U$ is a subset of the set that is covered at least twice by the union (\ref{kfold}). We claim that $\mu(\partial U)=0$. Indeed, if $\mu(\partial U)>0$ was true, then
\[ \mu(\mathcal{R}_\sigma(a))\leq
\sum_{(\gamma,b)\in T_{\rm ext}^{-k}(\V{0},a)}
\mu(\alpha^k(\mathcal{R}_\sigma(b)+\gamma))
-\mu(\partial U)\text{,}
\]
contradicting the measure disjointness of the union (\ref{kfold}). Thus $\mu(\partial U)=0$ and, hence, $\mu(\partial \mathcal{R}_\sigma(b))=0$. Since $b\in \mathcal{A}$ was arbitrary, we are done.
\end{proof}

\subsection{Adic transformation and domain exchange}\label{sec-adic}

Siegel \cite{Si:03} shows that, if $\sigma$ satisfies the strong coincidence condition (see Definition \ref{defscc}), the following hold:
\begin{enumerate}
\item The subtiles $\mathcal{R}_\sigma(a)$ are disjoint in measure.
\item $(X_\sigma,S)$ is isomorphic in measure to $(\mathcal{R}_\sigma,\mathcal{E})$, where $\mathcal{E}$ is the \emph{domain exchange} $\mathcal{E}(\V{z})=\V{z}+\Phi'(\delta(a))$, for $\V{z}\in\mathcal{R}_\sigma(a)$.
\end{enumerate}
We give a proof of the second result connecting it also to the \emph{adic transformation} $\hat{\Z}_\sigma\to\hat{\Z}_\sigma$, $x\mapsto x+\delta(w_0)$ on $\hat{\Z}_\sigma$ ({\it cf.} \cite[Proposition 2.3]{CS:01}).

We will need the following lemma (see \cite[Lemma 4.1, Proposition 5.1, Theorem 5.1]{CS:01a}).
\begin{lemma}\label{anlem}
Let $w\in X_\sigma$ and $E_\mathcal{P}(w)=(p_i,a_i,s_i)_{i\geq 0}$ its prefix-suffix development. Then $E_\mathcal{P}(\chi(w))=(p_i,a_i,s_i)_{i\geq 1}$ and $E_\mathcal{P}(\sigma(w))=(q_i,b_i,t_i)_{i\geq 0}$ is such that $q_0=\epsilon$ and $q_{i+1}=p_i$, for every $i\geq 0$. If $Sw$ is a periodic point of $\sigma$ then $E_\mathcal{P}(Sw)=(\epsilon,b_i,t_i)_{i\geq 0}$ and $E_\mathcal{P}(w)=(p_i,a_i,\epsilon)_{i\geq 0}$, with $(p_i)_{i\geq 0}$ periodic. If $Sw$ is not periodic for $\sigma$ then $E_\mathcal{P}(Sw)=(q_i,b_i,t_i)_{i\geq 0}$ is such that there exists an integer $k_0$ with $\sigma^k(p_k)\cdots\sigma^0(p_0)a_0=\sigma^k(q_k)\cdots\sigma^0(q_0)$, for all $k\geq k_0$. 
\end{lemma}
\begin{proof}[Sketch of the proof]
The first two statements follow from the definition of $E_\mathcal{P}$. For the third, a successor map $\psi$ defined on $X_\mathcal{P}^l$ and conjugate to the shift 
on $X_\sigma$ is introduced, \emph{i.e.}, such that $\psi(E_\mathcal{P}(w))=E_\mathcal{P}(Sw)$, for $w\in X_\sigma$. Given $(p_i,a_i,s_i)_{i\geq 0} \in X_\mathcal{P}^l$, 
$\psi((p_i,a_i,s_i)_{i\geq 0})=(q_i,b_i,t_i)_{i\geq 0}$ is defined as follows: let $i_0$ be the first index such that $s_{i_0}\neq\epsilon$; then, 
$(q_{i_0},b_{i_0},t_{i_0})$ is such that $q_{i_0}b_{i_0}t_{i_0}=p_{i_0}a_{i_0}s_{i_0}$ and $\abs{q_{i_0}}=\abs{p_{i_0}}+1$, for $i\geq i_0$, 
$(q_i,b_i,t_i)=(p_i,a_i,s_i)$, and for $0\leq i<i_0$ we take $(\epsilon,b_i,t_i)$ such that $\sigma(b_{i+1})=b_it_i$.
This is precisely an adic transformation, and, for its particular shape, we can deduce the last claim.
\end{proof}

\begin{proposition}
Let $\sigma$ be an irreducible Pisot substitution satisfying the strong coincidence condition. Let
\[
\varphi: X_\sigma \to \hat{\Z}_\sigma\text{,}\quad w \mapsto \sum_{i\geq 0}\delta(p_i)\alpha^i
 \]
where $E_\mathcal{P}(w)=(p_i,a_i,s_i)_{i\geq 0}\in\phantom{}^\omega\mathcal{P}$ is the prefix-suffix development of $w=\cdots w_{-1}.w_0w_1\cdots$. 
Then the action of $\sigma$ on $X_\sigma$ is conjugate to the multiplication by $\alpha$ on $\hat{\Z}_\sigma$ and the following diagram
\[
\xymatrix{ X_\sigma\; \ar[r]^{\varphi} \ar[d]_S & \;\hat{\Z}_\sigma\; \ar[r]^{\Phi'} \ar[d]^{+\delta(w_0)} & \;\mathcal{R}_\sigma \ar[d]^{+\Phi'(\delta(w_0))} \\  
X_\sigma\; \ar[r]^{\varphi} & \;\hat{\Z}_\sigma\; \ar[r]^{\Phi'} & \;\mathcal{R}_\sigma }
\]
is commutative.
\end{proposition}
\begin{proof}
The first statement follows from Lemma \ref{anlem} observing that the action of $\sigma$ on $X_\sigma$ is conjugate to the right extension of elements of $X_\mathcal{P}^l$ by an element that has an empty prefix.

The commutativity of the left diagram is also a consequence of Lemma \ref{anlem}. Let $w\in X_\sigma$ and $E_\mathcal{P}(w)=(p_i,a_i,s_i)_{i\geq 0}$.

If $Sw$ is not a periodic point of $\sigma$, $E_\mathcal{P}(Sw)=(q_i,b_i,t_i)_{i\geq 0}$ is such that there exists an integer $k_0$ with $\sigma^k(p_k)\cdots\sigma(p_0)w_0=\sigma^k(q_k)\cdots\sigma^0(q_0)$, for all $k\geq k_0$. Thus
\begin{align*}
\varphi(Sw) = \sum_{i\geq 0} \delta(q_i)\alpha^i = \sum_{i\geq 0} \delta(p_i)\alpha^i + \delta(w_0) = \varphi(w)+\delta(w_0)\text{.}
\end{align*}
If $Sw$ is a periodic point of $\sigma$ then $E_\mathcal{P}(Sw)=(\epsilon,b_i,t_i)_{i\geq 0}$ and $E_\mathcal{P}(w)=(p_i,a_i,\epsilon)_{i\geq 0}$, with $(p_i)_{i\geq 0}$ periodic with period $\ell$. Thus $\varphi(Sw)=0$ and $\sigma^{i+kl}(p_{i+k\ell})\cdots p_0w_0 = \sigma^{k\ell}(\sigma^i(p_i)\cdots p_0w_0)$
for every $i<\ell$ and every integer $k$. Therefore
\begin{align*}
\varphi(w) &= \lim_{k\to\infty}\sum_{i=0}^k \delta(p_i) \alpha^i = \lim_{k\to\infty} (\delta(p_k)\alpha^k + \cdots+\delta(p_1)\alpha+ \delta(p_0w_0)) -\delta(w_0) \\
&=  \lim_{k\to\infty} \alpha^{k\ell}(\delta(p_i)\alpha^i+\cdots+\delta(p_1)\alpha+\delta(p_0w_0))-\delta(w_0) \\ &= 0-\delta(w_0)=\varphi(Sw)-\delta(w_0)\text{.}
\end{align*}
The commutativity of the right diagram follows simply applying $\Phi'$ extended to elements of $\hat{\Z}_\sigma$ and observing that the addition by $\Phi'(\delta(w_0))$ is well-defined up to a set of measure zero.
\end{proof}

Observe that the adic transformation can be interpreted and computed by Bratteli diagrams (see {\it e.g.\ } \cite{Du:10}).

\end{section}

\begin{section}{Multiple tilings and tilings; property (F)}

In this section we show that the subtiles $\mathcal{R}_\sigma(a)$ induce a multiple tiling of $K_\sigma$ with respect to the translation set $\Gamma$. Moreover, we give a tiling criterion in terms of a finiteness condition of $(\sigma,a)$-expansions.

\subsection{Multiple tiling property}
We call a collection $\mathcal{C}$ of compact subsets of $K_\sigma$ a \emph{multiple tiling} of $K_\sigma$ if each element of $\mathcal{C}$ is the closure of its interior and if there exists a
positive integer $m$ such that $\mu$-almost every point of $K_\sigma$ is contained in exactly $m$ elements
of $\mathcal{C}$. If $m = 1$ then $\mathcal{C}$ is called a \emph{tiling} of $K_\sigma$.

In this section we will prove that every irreducible Pisot substitution induces a multiple tiling of the associated representation space.

A \emph{patch} is defined as a finite subset of $\Gamma$.
We say that $\Gamma$ is \emph{repetitive} (or \emph{quasi-periodic}) if for any patch $P$ there exists a radius $R>0$ such that every ball of radius $R$ in $\Gamma$ contains a translate of $P$.

\begin{lemma}\label{lem:rep}
The translation set $\Gamma$ is repetitive.
\end{lemma}

This result is already contained in \cite[Proposition~6.72]{Sin:06} in the CPS setting. We present a similar proof in the spirit of \cite[Theorem~5.3.13]{BST:10} in our setting.

\begin{proof}
Let $P=\{(\gamma_k,a_k)\text{, }1\leq k\leq \ell\}$ be a patch of $\Gamma$. We can write each $\gamma_k$ as $\Phi'(x_k)$, for $x_k\in V\cdot\Z[\alpha^{-1}]\cap[0,\delta(a_k))$. Let $R_1$ be such that $B(0,R_1)$ contains the patch $P$. There exists $\varepsilon_k>0$ such that $x_k\in V\cdot\Z[\alpha^{-1}]\cap[0,(1-\varepsilon_k)\delta(a_k))$, for each $1\leq k\leq \ell$.
Set $\varepsilon:=\frac{1}{2}\min_k\varepsilon_k\delta(a_k)$. Then $\Phi'(x)+P$ is in $\Gamma$, for every $x\in V\cdot\Z[\alpha^{-1}]\cap[0,\varepsilon)$.

It remains to prove that there exists $R>0$ such that any ball of radius $R$ in $K_\sigma$ contains a point $\Phi'(x)$ with $x\in V\cdot\Z[\alpha^{-1}]\cap[0,\varepsilon)$. By the denseness of $V\cdot\Z[\alpha^{-1}]$ in $\R$ there exists $x_0\in V\cdot\Z[\alpha^{-1}]\cap[0,\varepsilon/2)$. Let $\Delta=\max_{a\in\mathcal{A}} \delta(a)$.
We can divide $[0,\Delta)$ in $N=\lceil 2\Delta/\varepsilon\rceil$ subintervals $[j\varepsilon/2,(j+1)\varepsilon/2)$ of length $\varepsilon/2$. For each $j\leq N$, there exists $m_j\in\Z$ such that $m_jx_0+[j\varepsilon/2,(j+1)\varepsilon/2)\subset [0,\varepsilon)$.

Fix a point $\eta\in K_\sigma$. We know by Lemma \ref{del} that there is $R_2>0$ such that every ball of radius $R_2$ contains at least one element of $\Gamma$. In particular, the ball $B(\eta,R_2)$ contains a point $\Phi'(x)$ with $x\in V\cdot\Z[\alpha^{-1}]\cap[0,\Delta)$. Thus there exists $j\in\{0,\ldots, N\}$ such that $x\in V\cdot\Z[\alpha^{-1}]\cap[j\varepsilon/2,(j+1)\varepsilon/2)$, and, hence, $m_j\in\mathbb{Z}$ such that $m_jx_0+x\in [0,\varepsilon)$. This implies that $\Phi'(x+m_jx_0)+P$ occurs in $\Gamma$.

Therefore, the ball centered in $\eta$ with radius $R:=R_1+R_2+\max_j\norm{\Phi'(m_jx_0)}$ contains a translated copy of the patch $P$ and the lemma is proved.
\end{proof}

We are now in a position to state the multiple tiling result (see also \cite[Theorem~5.3.13]{BST:10} for irreducible unit substitutions).

\begin{theorem}\label{multi}
Let $\sigma$ be an irreducible Pisot substitution. The collection $\mathcal{C}_\sigma=\{\mathcal{R}_\sigma(a)+\gamma: (\gamma,a)\in \Gamma\}$ is a multiple tiling of $K_\sigma$.
\end{theorem}

\begin{proof}
Assume that the assertion of the theorem is false. Then there exist $\ell_1, \ell_2 \in \mathbb{N}$, $\ell_1 < \ell_2$, and $M_1, M_2 \subset K_\sigma$ with $\mu(M_i)>0$ such that each element of $M_i$ is covered exactly $\ell_i$ times by the elements of $\mathcal{C}_\sigma$ ($i=1,2$). As the boundaries of the subtiles have zero measure by Theorem~\ref{basictheorem}~(\ref{basicthree}), there exist points $\eta_i \in M_i$  that are not contained in the boundary of any element of $\mathcal{C}_\sigma$. Thus we can find $\varepsilon>0$ such that $B(\eta_i,\varepsilon)$ is covered exactly $\ell_i$ times by the collection $\mathcal{C}_\sigma$. This implies that there exist a  patch $P_2\subset\Gamma$ with $\ell_2$ elements such that
\[
B(\eta_2,\varepsilon)\subset \bigcap_{(\gamma,a)\in P_2}\mathcal{R}_\sigma(a)+\gamma\text{.}
 \]
Consider the inflated ball $\alpha^{-k}B(\eta_1,\varepsilon)$. By the same arguments presented above, each point of $\alpha^{-k}B(\eta_1,\varepsilon)$ is covered by exactly $\ell_1$ tiles of the collection $\alpha^{-k}\mathcal{C}_\sigma$.
Each of the inflated tiles of $\alpha^{-k}\mathcal{C}_\sigma$ can be decomposed in a finite union of tiles in $\mathcal{C}_\sigma$ which are pairwise disjoint in measure. Thus almost each point in $\alpha^{-k}B(\eta_1,\varepsilon)$ is contained in exactly $\ell_1$ tiles of $\mathcal{C}_\sigma$. By Lemma~\ref{lem:rep} we can pick a suitable large $k$ such that $\alpha^{-k}B(\eta_1,\varepsilon)$ contains a translated copy $P_2+\gamma$, for some $\gamma\in\Gamma$. Therefore $B(\eta_2,\varepsilon)+\gamma$ is contained in $\alpha^{-k}B(\eta_1,\varepsilon)$, for $k$ large enough. The ball $B(\eta_2,\varepsilon)$ is covered exactly $\ell_2$ times, consequently $B(\eta_2,\varepsilon)+\gamma$ is covered at least $\ell_2$ times, but this yields a contradiction since almost every point in $\alpha^{-k}B(\eta_1,\varepsilon)$ is contained in exactly $\ell_1$ tiles, and $\ell_1<\ell_2$.
\end{proof}

We state the following conjecture (see also the recent paper by Barge {\it et al.}~\cite{BBJS:12} where it is proposed to replace the irreducibility condition by a topological condition on the tiling space associated with $\sigma$).

\begin{conjecture}
Every irreducible Pisot substitution $\sigma$ induces a self-replicating tiling of its associated representation space $K_\sigma$.
\end{conjecture}

\subsection{Finiteness property}

In this section we provide a tiling criterion for $\mathcal{C}_\sigma$ based on the \emph{geometric property (F)}. We take inspiration mainly from \cite{ST:09}.

Consider the set
\begin{equation*}
\mathcal{U}:=\bigcup_{a\in\mathcal{A}} (\V{0},a)\subset \Gamma.
\end{equation*}
It is easy to see that $\mathcal{U}\subseteq \T(\mathcal{U})$: indeed, $(\V{0},b)\in \T(\V{0},a)$ if $\sigma(b)=as$, {\it i.e.}, if $p=\epsilon$.
Thus $(\V{0},b)\in \T(\V{0},a)$ where $a$ is the first letter of $\sigma(b)$.
Hence the sequence $((\T)^m(\mathcal{U}))_{m\geq 0}$ is an increasing sequence of subsets of $\Gamma$.

\begin{definition}
Let $\sigma$ be an irreducible Pisot substitution. We say that the substitution $\sigma$ satisfies the
\emph{geometric property (F)} if the iterations of
$\T$ on $\mathcal{U}$ eventually cover the whole self-replicating translation set $\Gamma$, \emph{i.e.}, if
\begin{equation*}
\Gamma=\bigcup_{m\geq 0}T_{\rm ext}^{-m}(\mathcal{U})\text{.}
\end{equation*}
\end{definition}
The geometric property (F) is an equivalent formulation of the finiteness property firstly introduced in \cite{FS:92} in the beta-numeration framework and further studied in \cite{A:00}.
Here we shall interpret it as a finiteness condition on $(\sigma,b)$-expansions.
Indeed, given $(\gamma,b)\in\Gamma$, $\gamma$ can be written as $\Phi'(x)$ where $x\in V\cdot\Z[\alpha^{-1}]\cap[0,\delta(b))$ which has a unique $(\sigma,b)$-expansion by Proposition \ref{DT}, \emph{i.e.},
$x=\sum_{i\geq 1}\delta(p_i)\alpha^{-i}$. Then we can say as well that $(\gamma,b)$ has a formal $(\sigma,b)$-expansion in $K_\sigma$, namely $\gamma=\sum_{i\geq 1}\Phi'(\delta(p_i)\alpha^{-i})$.
\begin{proposition}
The substitution $\sigma$ satisfies the geometric property (F) if and only if every point $(\gamma,b)\in\Gamma$ has a unique finite $(\sigma,b)$-expansion.
\end{proposition}

\begin{proof}
Let $(\gamma,b)\in\Gamma$. If property (F) holds, then $(\gamma,b)\in T_{\rm ext}^{-m}(\V{0},a)$, for some $m\geq 0$ and $a\in\mathcal{A}$. Thus, using (\ref{E1fold}) we get
\begin{equation}\label{hary} \gamma=\alpha^{-m}\Phi'(\delta(p_0))+\alpha^{-m+1}\Phi'(\delta(p_1))+\cdots+\alpha^{-1}\Phi'(\delta(p_{m-1}))\text{,} \end{equation}
where $b\stackrel{p_{m-1}}{\longrightarrow} \cdots\stackrel{p_1}{\longrightarrow}a_1\stackrel{p_0}{\longrightarrow} a_0$ is a walk in the prefix automaton ending at $a=a_0$.

On the other hand, suppose that $(\gamma,b)\in\Gamma$ has a unique finite $(\sigma,b)$-expansion
$\gamma=\alpha^{-1}\Phi'(\delta(p_1))+\cdots+\alpha^{-m}\Phi'(\delta(p_m))$ with $b\stackrel{p_1}{\longrightarrow} \cdots\stackrel{p_{m-1}}{\longrightarrow}a_{m-1}\xrightarrow{p_m} a$. This yields that $\alpha^m\gamma\in\mathcal{R}_\sigma(a)$ and using the iterated set equation in (\ref{kfold}) we get
\[
\gamma\in \bigcup_{(\eta,c)\in T_{\rm ext}^{-m}(\V{0},a)}\mathcal{R}_\sigma(c)+\eta\text{.}
\]
Thus we may conclude that $(\gamma,b)\in T_{\rm ext}^{-m}(\V{0},a)$.
\end{proof}

For an irreducible Pisot substitution $\sigma$ satisfying the geometric property (F), it is immediate from Proposition \ref{coverprop} and the definition of subtiles that every $\V{z}\in K_\sigma$ admits a $(\sigma,a)$-expansion in $K_\sigma$ for some $a\in\mathcal{A}$ ({\it cf.}  \cite[Proposition~3.9]{ST:09}), {\it i.e.},
\[
\V{z}=\sum_{i=m}^\infty \Phi'(\delta(p_i)\alpha^i)\text{, }\quad m\in\Z\text{.}
\]

In the context of beta-numeration, Akiyama \cite{A:02} proved that property (F) is equivalent to the fact that $\V{0}$ is an \emph{exclusive} inner point of the central tile. Our next aim is to carry over this statement to the substitution context.

\begin{definition}
A point $\eta\in K_\sigma$ is an \emph{exclusive} inner point of a
collection of finitely many tiles in a multiple tiling if it is contained exclusively
in this collection and in no other tile of the multiple tiling.
\end{definition}

\begin{definition}
The \emph{zero-expansion graph} $\mathcal{G}^{(0)}$ of $\sigma$ is the directed graph such that the following conditions hold.
\begin{itemize}
\item The nodes $(\gamma,a)\in \Gamma$ are such that $\norm{\gamma}\leq M$, where $M$ is taken as in Equation (\ref{ball}).
\item There is a directed edge from $(\gamma_1,a_1)$ to $(\gamma_2,a_2)$ if and only if $(\gamma_2,a_2)\in \T(\gamma_1,a_1)$.
\item Every node is the starting point of an infinite walk.
\end{itemize}
\end{definition}
The zero-expansion graph is used to characterize all the elements $(\gamma,a)\in\Gamma$ for which the tile $\mathcal{R}_\sigma(a)+\gamma$ contains $\V{0}$. Suppose $\V{0}\in\mathcal{R}_\sigma(a)+\gamma$. This implies that $\gamma\in B(\V{0},M)$, where $M$ is as in (\ref{ball}).

\begin{proposition}
The zero-expansion graph $\mathcal{G}^{(0)}$ of an irreducible Pisot substitution $\sigma$ is well defined and finite. A pair $(\gamma,a)$ is a node of this graph if and only if $\mathbf{0}\in\mathcal{R}_\sigma(a)+\gamma$.
\end{proposition}

\begin{proof}
The graph is finite since the nodes are elements of the Delone set $\Gamma$ with bounded norm.
Consider a node $(\gamma,a)=(\gamma_0,a_0)\in\mathcal{G}^{(0)}$ and the infinite walk $\{(\gamma_k,a_k)\}_{k\geq 0}$ starting from it. Then, by definition of edges, we get a left-infinite walk in the prefix automaton $\cdots\stackrel{p_2}{\longrightarrow}a_2\stackrel{p_1}{\longrightarrow}a_1\stackrel{p_0}{\longrightarrow}a_0$ and
\[ \gamma=-\Phi'(\delta(p_0))-\alpha\Phi'(\delta(p_1))
-\cdots-\alpha^k\Phi'(\delta(p_k))+\alpha^{k+1}\gamma_{k+1}\text{.} \]
Since multiplication by $\alpha$ is a contraction and $\norm{\gamma_k}$ is uniformly bounded in $k$, we obtain for $k\rightarrow\infty$ a convergent power series: $\gamma=-\sum_{k\geq 0}\alpha^k\Phi'(\delta(p_k))$.
Thus $-\gamma\in\mathcal{R}_\sigma(a)$ and hence $\V{0}\in\mathcal{R}_\sigma(a)+\gamma$.

Suppose conversely that $\V{0}\in\mathcal{R}_\sigma(a)+\gamma$, for $(\gamma,a)\in\Gamma$. Then 
$\gamma=-\sum_{k\geq 0}\alpha^k\Phi'(\delta(p_k))$, where $(p_k)_{k\geq 0}$ is the labeling of a left-infinite walk in the prefix automaton ending at state $a$.
Let $\gamma_\ell=-\sum_{k\geq 0}\alpha^k\Phi'(\delta(p_{k+\ell}))$. Each $\gamma_\ell\in B(\V{0},M)$ and $\alpha\gamma_{\ell+1}=\gamma_\ell+\Phi'(\delta(p_\ell))$, \emph{i.e.}, $(\gamma_{\ell+1},a_{\ell+1})\in \T(\gamma_\ell,a_\ell)$.
By induction $(\gamma_\ell,a_\ell)\in\Gamma$ for all $\ell\in\N$, since this holds for $\gamma=\gamma_0$ and $\Gamma$ is invariant under $\T$. Hence, $(\gamma_k,a_k)_{k\geq 0}$ is an infinite walk in the zero-expansion graph starting from $(\gamma,a)$.
\end{proof}

\begin{lemma}
Let $\sigma$ be an irreducible Pisot substitution that satisfies the strong coincidence condition. Then $\sigma$ satisfies the geometric property (F) if and only if
$\mathbf{0}$ is an exclusive inner point of the central tile $\mathcal{R}_\sigma$.
\end{lemma}

\begin{proof}
Suppose that $\V{0}$ is not an exclusive inner point of $\mathcal{R}_\sigma$. Then there exists $\gamma\neq \V{0}$, which has a finite expansion by property (F), such that $\V{0}\in\mathcal{R}_\sigma(a)+\gamma$, which implies
$\V{0}=\sum_{j=-m}^{\infty}\alpha^j\Phi'(\delta(p_j))$, for $m\in\N$. Multiplying by $\alpha^{-k}$ yields $\V{0}=\sum_{j=-m}^{\infty}\alpha^{j-k}\Phi'(\delta(p_j))$, that means
$\V{0}\in\mathcal{R}_\sigma(a_k)+\sum_{\ell=1}^{m+k}\alpha^{-\ell}\Phi'(\delta(p_{k-\ell}))$ for each $k\in\N$, where each of these sums represent a different element since the representation is unique. This gives a contradiction with the local finiteness of the covering. Therefore $\V{0}$ is an exclusive inner point.

Assume that (F) does not hold, \emph{i.e.}, there exists $(\gamma_0,a_0)\in\Gamma\setminus\bigcup_{m\geq 0}
T_{\rm ext}^{-m}(\mathcal{U})$. In particular $\gamma_0\neq \V{0}$. Since $\T(\Gamma)=\Gamma$, we can define a sequence $\{(\gamma_k,a_k)\}_{k\geq 1}$ of elements of $\Gamma$ with
\[ (\gamma_k,a_k)\in \T(\gamma_{k+1},a_{k+1})\text{, }\quad k\geq 0\text{.} \]
Since multiplication by $\alpha$ is a contraction in $K_\sigma$, for some $k_0\in\N$ large enough, $\gamma_k\in B(\V{0},M)$, for all $k\geq k_0$, where $M$ is as in equation (\ref{ball}).
There exist only finitely many $(\gamma_k,a_k)\in\Gamma$ such that $\gamma_k\in B(\V{0},M)$, since $\Gamma$ is a Delone set. Then
\[ \exists k'>k_0,\; \exists \ell>0\quad\text{ such that }\quad (\gamma_{k'},a_{k'})=(\gamma_{k'+\ell},a_{k'+\ell})\text{,} \]
and $\gamma_{k'}\neq \V{0}$, otherwise $\gamma_0\in\bigcup_{m\geq 0}T_{\rm ext}^{-m}(\mathcal{U})$.
This is equivalent to the existence of a loop in the zero-expansion graph $\mathcal{G}^{(0)}$
\[
\gamma_{k'}\rightarrow\gamma_{k'+\ell-1}\rightarrow\cdots \rightarrow\gamma_{k'+1}\rightarrow\gamma_{k'}\text{,}
\]
and, by the definition of $\mathcal{G}^{(0)}$, this implies that
$\V{0}\in \mathcal{R}_\sigma(a_{k'})+\gamma_{k'}$. Since $\gamma_{k'}\neq \V{0}$, we have that $\V{0}$ is not an exclusive inner point of $\mathcal{R}_\sigma$.
\end{proof}

Finally we can generalize the tiling condition given in \cite{ABBS:08} for beta numeration.

\begin{theorem}
Let $\sigma$ be an irreducible Pisot substitution. If $\sigma$ satisfies the geometric property (F) and the strong coincidence condition, the self-replicating multiple tiling $\{\mathcal{R}_\sigma(a)+\gamma : (\gamma,a)\in\Gamma\}$ is a tiling.
\end{theorem}

\begin{proof}
By the geometric property (F) $\V{0}$ is an exclusive inner point. Since the strong coincidence condition holds, the subtiles $\mathcal{R}_\sigma(a)$ are disjoint in measure (see the end of Section \ref{sec-adic}) and there is a set of positive measure around $\V{0}$ which is covered only once. Since we know by Theorem \ref{multi} that $\{\mathcal{R}_\sigma(a)+\gamma : (\gamma,a)\in\Gamma\}$ is a multiple tiling, this implies that the covering degree is $1$.
\end{proof}
\end{section}

\begin{section}{Examples}\label{Sec:Ex}
In this section we consider two examples of irreducible non-unit Pisot substitutions.
\subsection{A two letters example}
Consider the substitution $\sigma(1)=1^5 2$, $\sigma(2)=1^3$. We have
\[ M_\sigma=\begin{pmatrix} 5 & 3 \\ 1 & 0 \end{pmatrix}\text{, }\qquad
\det(xI-M_\sigma)=x^2-5x-3\text{.} \]
The dominant eigenvalue is $\alpha=\frac{5+\sqrt{37}}{2}$ and its conjugate $\bar{\alpha}$ satisfies $\abs{\bar{\alpha}}<1$.

\begin{figure}[h]
\centering
\includegraphics[scale=.4]{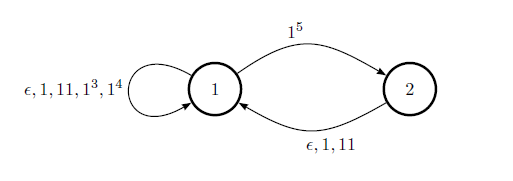}
\caption{Prefix automaton of $\sigma$.}
\end{figure}

For every prime divisor $p$ of $N(\alpha)=-3$ we have a prime $\mathfrak{p}\mid(p)$ such that $\abs{\alpha}_\mathfrak{p}<1$.
Therefore we factorize the prime ideal $(3)$ in $\mathcal{O}$ and we find
\[ (3)=(3,\alpha)(3,\alpha+1)=\underbrace{(\alpha)}_{\mathfrak{p}_1}\underbrace{(5-\alpha)}_{\mathfrak{p}_2}\text{,}
\]
\emph{i.e.}, $(3)$ splits completely in $\mathcal{O}$: $e_i=f_i=1$, for $i=1,2$. Now we look at the normalized absolute values $\abs{\alpha}_{\mathfrak{p}_1}=\frac{1}{3}$, $\abs{\alpha}_{\mathfrak{p}_2}=1$,
and we deduce that we have to consider only the non-Archimedean completion $K_{\p_1}$, which is an extension of degree $e_1f_1=1$ of $\mathbb{Q}_3$, equipped with the normalized absolute value $\abs{\,\cdot\,}_{\mathfrak{p}_1}$. Thus the representation space is $K_\sigma=\R\times \Q_3$. Notice that $\alpha$ is a uniformiser for $K_{\p_1}$ and we can represent each element of $\Q_3$ as $\sum_{i=m}^\infty d_i\alpha^i$ with $d_i\in\{0,1,2\}$, $m\in\Z$.
The canonical embedding is given explicitly by
\[
\Phi':\mathbb{Q}(\alpha)\longrightarrow \mathbb{R}\times\mathbb{Q}_3\text{, }\quad
a_0+a_1\alpha\longmapsto \Big(a_0+a_1\bar{\alpha},\sum_{i=m}^\infty d_i\alpha^i\Big)\text{.}
\]
We choose $\textbf{v}_\alpha=(\frac{\alpha}{3},1)$ as left eigenvector of $M_\sigma$. With this choice, we get the following set of digits for the Dumont-Thomas expansions:
\[ \mathcal{D}=\{\delta(\epsilon),\delta(1),\delta(11),\delta(1^3),\delta(1^4),\delta(1^5)\}
=\bigg\{0,\frac{\alpha}{3},\frac{2\alpha}{3},\alpha,\frac{4\alpha}{3},\frac{5\alpha}{3}\bigg\}\text{.}
\]
We obtain the central tile $\mathcal{R}_\sigma$ by taking the closure of the embedding of the $\sigma$-integers. Having chosen $\textbf{v}_\alpha$ as above, we get that $\mathcal{R}_\sigma\subset\mathbb{R}\times\mathbb{Z}_3$
(see Proposition \ref{pheight}). Furthermore, the subtiles satisfy the following set equations:
\begin{align*}
\mathcal{R}_\sigma(1)&=\alpha\mathcal{R}_\sigma(1)+(\alpha\mathcal{R}_\sigma(1)+
\Phi'(\delta(1)))+(\alpha\mathcal{R}_\sigma(1)+\Phi'(\delta(11))) \\
&\quad+(\alpha\mathcal{R}_\sigma(1)+\Phi'(\delta(1^3)))+(\alpha\mathcal{R}_\sigma(1)+\Phi'(\delta(1^4))) \\ &\quad+\alpha\mathcal{R}_\sigma(2)+
(\alpha\mathcal{R}_\sigma(2)+\Phi'(\delta(1)))+(\alpha\mathcal{R}_\sigma(2)+\Phi'(\delta(11)))\text{,} \\
\mathcal{R}_\sigma(2)&=\alpha\mathcal{R}_\sigma(1)+\Phi'(\delta(1^5))\text{.}
\end{align*}

\begin{figure}[h]
\begin{center}$
\begin{array}{cc}
\includegraphics[scale=.4]{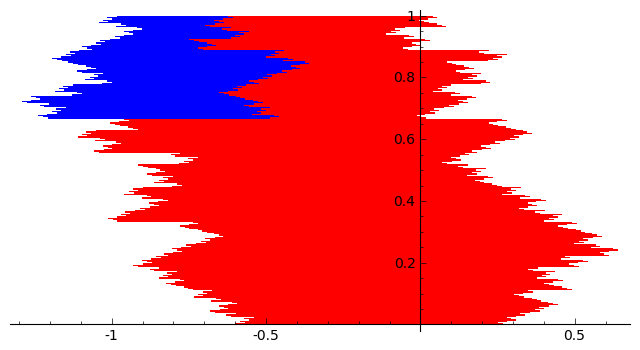} &
\includegraphics[scale=.4]{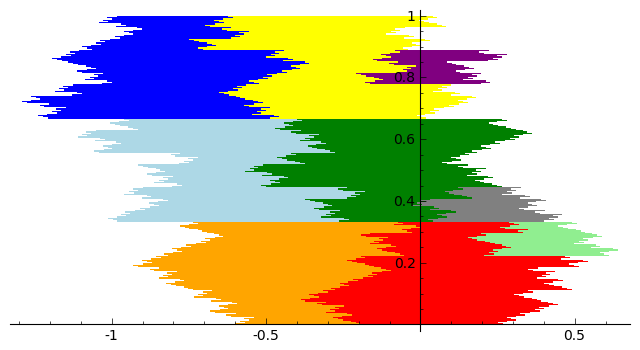}
\end{array}$
\end{center}
\setlength\unitlength{1mm}
\begin{picture}(0,0)
\put(-5,0){$\mathbb{R}$}
\put(-30,38){$\mathbb{Z}_3$}
\put(65,0){$\mathbb{R}$}
\put(39,38){$\mathbb{Z}_3$}
\end{picture}
\caption{The central tile $\mathcal{R}_\sigma$ divided in the red (light gray) subtile $\mathcal{R}_\sigma(1)$ and the blue (dark gray) subtile $\mathcal{R}_\sigma(2)$, and the self similar structure arising from the set equations.\label{fig1}}
\end{figure}
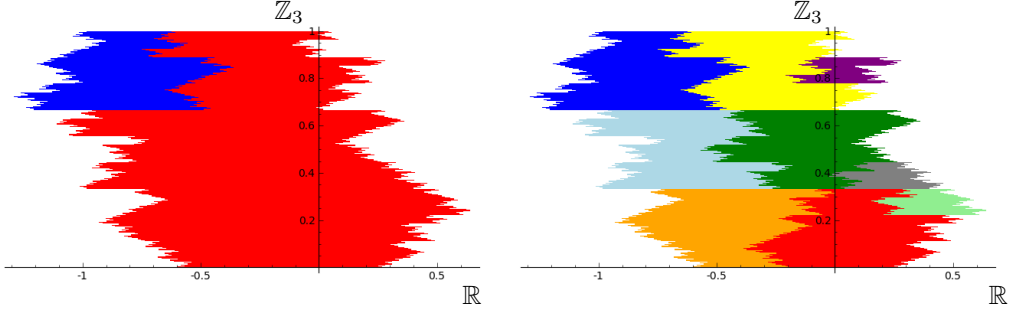

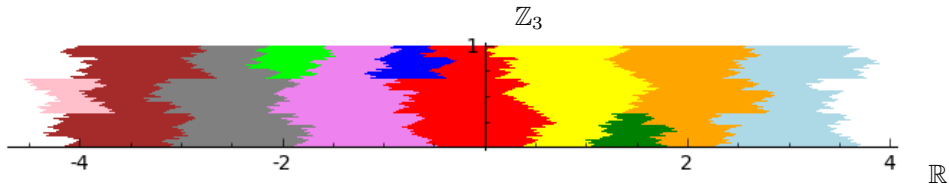
\begin{figure}[h]
\centering
\includegraphics[scale=.75]{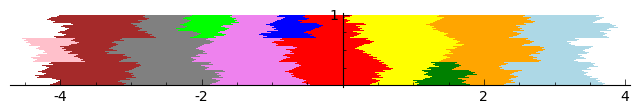}
\setlength\unitlength{1mm}
\begin{picture}(0,0)
\put(0,0){$\mathbb{R}$}
\put(-55,21){$\mathbb{Z}_3$}
\end{picture}
\caption{First $3$-adic level of the tiling of $K_\sigma$.}\label{figlev}
\end{figure}

In Figure \ref{figlev} it is represented the first $3$-adic level of the tiling, \emph{i.e.}, those tiles whose $\p$-adic part is contained in $\mathbb{Z}_3$: they are the $\mathcal{R}_x$ with
\[ x\in\bigg\{\frac{2\alpha}{3}-3,\frac{2\alpha}{3}-2,\frac{\alpha}{3}-1,0,1,2-\frac{\alpha}{3},
3-\frac{\alpha}{3}\bigg\}\text{.} \] One can check in fact that all these $x$ are $3$-adic integers. Furthermore the tiles $\mathcal{R}_x$, for $x=\frac{2\alpha}{3}-3,\frac{\alpha}{3}-1,0,3-\frac{\alpha}{3}$, are union of the two subtiles $\mathcal{R}_\sigma(1)$ and $\mathcal{R}_\sigma(2)$, because these $x$ are less than $1$, \emph{i.e.}, they are both in $[0,\delta(1))=[0,\alpha/3)$ and $[0,\delta(2))=[0,1)$.

\begin{figure}[h]
\centering
\includegraphics[scale=.7]{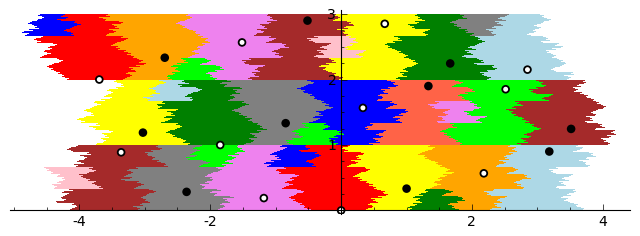}
\setlength\unitlength{1mm}
\begin{picture}(0,0)
\put(0,0){$\mathbb{R}$}
\put(-54,42){$\mathbb{Q}_3$}
\end{picture}
\caption{Tiling of the representation space $K_\sigma$ with translation set $\Gamma$. The black (white) points belong to $\Phi'(\text{Frac}(\sigma,1))$ (respectively $\Phi'(\text{Frac}(\sigma,2))$).}
\end{figure}
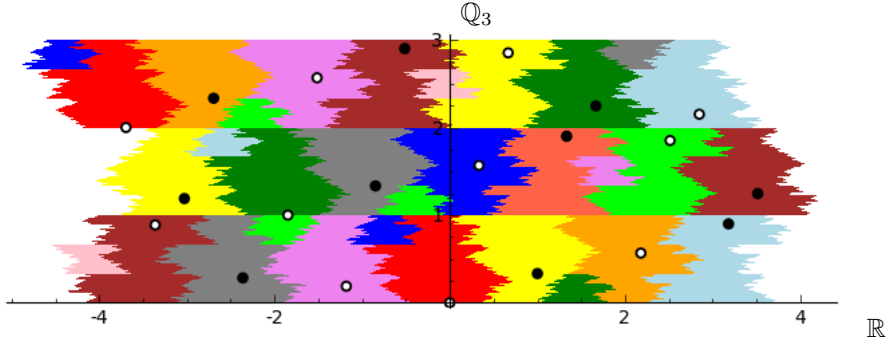

Note that for the representation of the tiles we used the following Euclidean model for $\mathbb{Q}_3$:
\[ g:\Q_3 \to \R^+\text{,}\quad
\sum_{i=m}^\infty d_i \alpha^i \mapsto \sum_{i=m}^\infty d_i 3^{-i-1}\text{.} \]
This map is onto, continuous, preserves the Haar measure but is not a homomorphism with respect to the addition.

In Figure \ref{figE1} and Figure \ref{figE2} it is represented the action of $\T$ on the basic faces $(\V{0},1)$, $(\V{0},2)$ and on their union $\mathcal{U}$. This is an example of substitution satisfying the geometric property (F).
Finally we show in Figure \ref{figex} the exchange of domains given by the substitution $\sigma$.

\begin{figure}[h]
\begin{subfigure}[h]{0.45\textwidth}
\centering
\includegraphics[scale=.45]{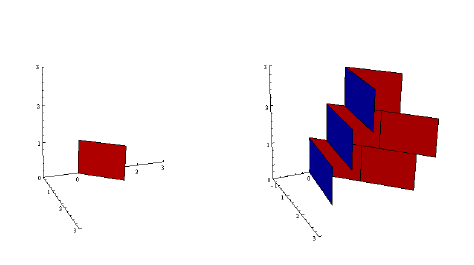}
\caption{The face $(\V{0},1)$ and its image $\T(\V{0},1)$.}
\end{subfigure}\qquad\quad
\begin{subfigure}[h]{0.45\textwidth}
\centering
\includegraphics[scale=.45]{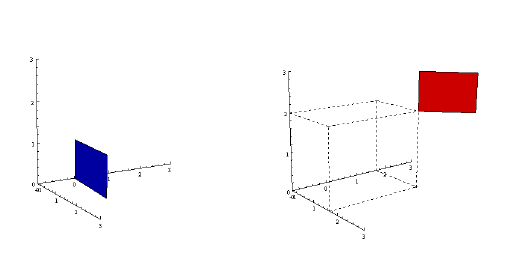}
\caption{The face $(\V{0},2)$ and its image $\T(\V{0},2)$.}
\end{subfigure}
\caption{}\label{figE1}
\end{figure}
\begin{figure}[h]
\begin{center}$
\begin{array}{cc}
\includegraphics[scale=.4]{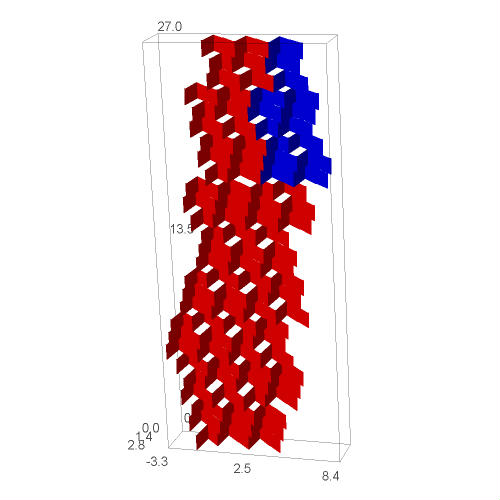} &
\includegraphics[scale=.37]{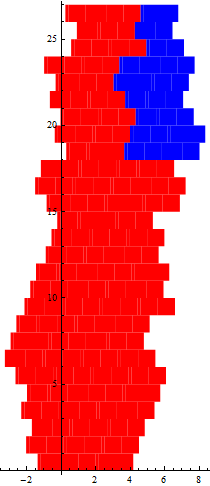}
\end{array}$
\end{center}
\caption{$T_{\rm ext}^{-3}(\mathcal{U})$ and its projection into $K_\sigma$.}\label{figE2}
\end{figure}

\begin{figure}[h]
\begin{center}$
\begin{array}{cc}
\includegraphics[scale=.35]{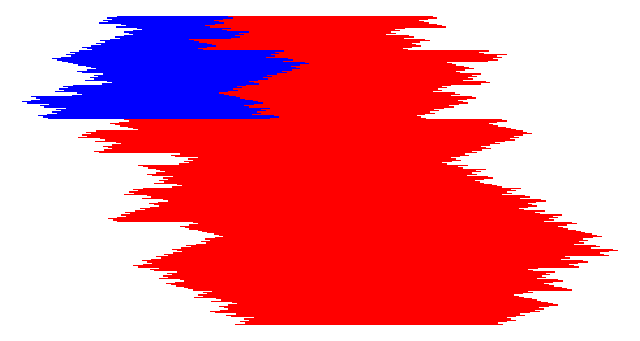} &
\includegraphics[scale=.35]{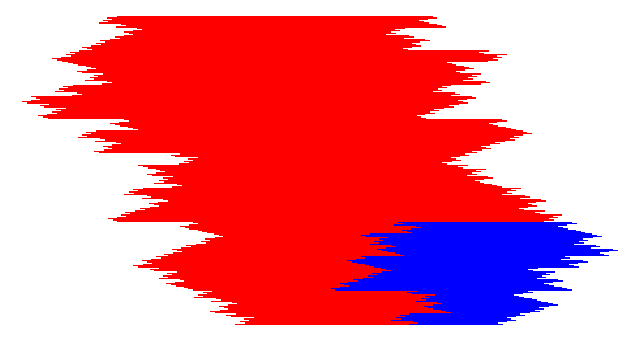}
\end{array}$
\end{center}
\caption{Exchange of domains.}\label{figex}
\end{figure}

\subsection{A three letters example}
Let $\sigma$ be the substitution $\sigma:1\mapsto 1113$, $2\mapsto 11$, $3\mapsto 2$.
The incidence matrix of the substitution and its characteristic polynomial are
\[ M_\sigma=\begin{pmatrix} 3 & 2 & 0 \\ 0 & 0 & 1 \\ 1 & 0 & 0 \end{pmatrix}\text{, }\qquad
\det(xI-M_\sigma)=x^3-3x^2-2\text{.} \]
One can check that the characteristic polynomial is irreducible over $\mathbb{Q}$ and has one real root $\alpha\approx 3.196$ and two complex conjugate roots
$\alpha_2,\bar{\alpha}_2\approx -0.098\pm 0.785 i$ with norm less than $1$.
\begin{figure}[h]
\centering
\includegraphics[scale=.5]{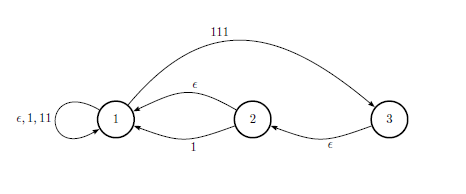}
\caption{Prefix automaton of $\sigma$.}
\end{figure}
We want to determine the representation space for the substitution $\sigma$.
Setting $K=\Q(\alpha)$, we know that the Archimedean part of the representation space is $\mathbb{C}$, while for the non-Archimedean part we have to compute the prime ideal factorization of $2\mathcal{O}$:
\[ (2)=(2,\alpha)^2\cdot(2,\alpha-1)=
\underbrace{(\alpha)^2}_{\mathfrak{p}_1^2}\cdot\underbrace{(-1-\alpha^2)}_{\mathfrak{p}_2}\text{.} \]
We have $\abs{\alpha}_{\mathfrak{p}_1}=\frac{1}{2}$, $\abs{\alpha}_{\mathfrak{p}_2}=1$, hence the non-Archimedean completion we have to consider is $K_{\mathfrak{p}_1}$, which is an extension of degree $e_1f_1=2$ of $\mathbb{Q}_2$.

There exist only $7$ non-isomorphic quadratic extensions of $\mathbb{Q}_2$, and one can check that $K_{\mathfrak{p}_1}\cong \Q_2(\sqrt{7})$. By the way, $\alpha$ is a uniformiser in $K_{\p_1}$, thus we can express every element of this completion as $\sum_{i=m}^\infty d_i \alpha^i$, with $d_i\in\{0,1\}$ and some $m\in\Z$.
Hence the canonical embedding is
\[ \Phi':\Q(\alpha)\to\mathbb{C}\times\mathbb{Q}_2(\sqrt{7})\text{,}\quad
a_0+a_1\alpha+a_2\alpha^2 \mapsto \bigg(a_0+a_1\alpha_2+a_2\alpha_2^2,\sum_{i=m}^\infty d_i \alpha^i\bigg)\text{.} \]
We represent each element of $\Q_2(\sqrt{7})$ with the Euclidean model $\Q_2(\sqrt{7})\to\R^+$,
$\sum_{i=m}^\infty d_i \alpha^i\mapsto \sum_{i=m}^\infty  d_i 2^{-i-1}$. In Figure \ref{fig3ex} the tiles associated to the substitution are represented. We choose $\V{v}_\alpha=(\alpha^2/2,\alpha,1)$ as left eigenvector of $M_\sigma$. With this choice we have the set of digits $\mathcal{D}=\{0,\alpha^2/2,\alpha^2,3\alpha^2/2\}$ for the Dumont-Thomas expansions. Furthermore we have that the $\p$-adic part of the central tile is contained in $\Z_2[\sqrt{7}]$.
\begin{figure}[h]
\centering
\includegraphics[scale=.35]{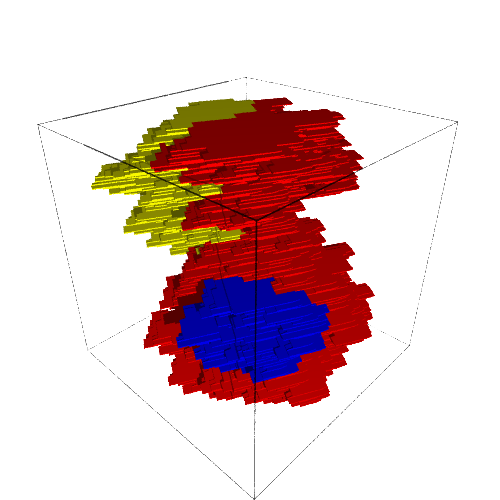}
\setlength\unitlength{1mm}
\begin{picture}(0,0)
\put(-20,5){$\C$}
\put(-32,55){$\Z_2[\sqrt{7}]$}
\end{picture}
\end{figure}

\begin{figure}[h]
\begin{center}$
\begin{array}{ccc}
\includegraphics[scale=.35]{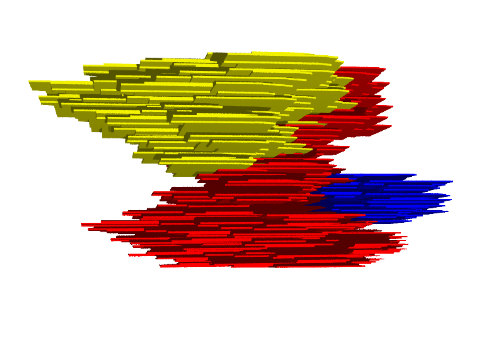} &
\includegraphics[scale=.35]{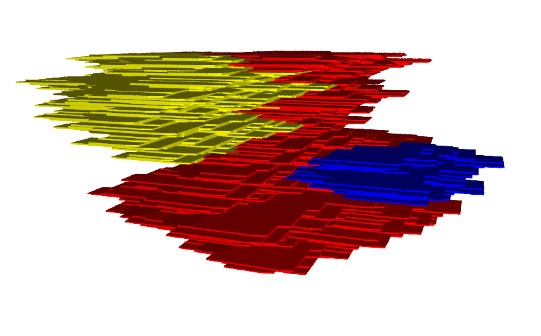} &
\includegraphics[scale=.35]{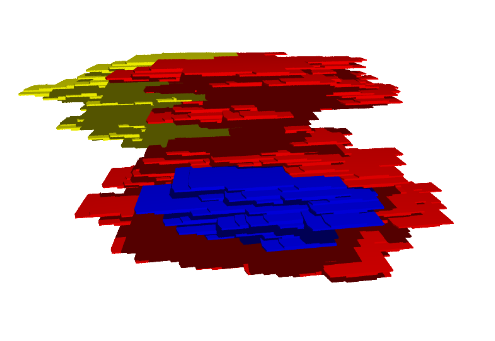}
\end{array}$
\end{center}
\caption{Pictures of the central tile $\mathcal{R}_\sigma$ divided in the red (gray) subtile $\mathcal{R}_\sigma(1)$, the blue (dark gray) $\mathcal{R}_\sigma(2)$ and the yellow (light gray) $\mathcal{R}_\sigma(3)$.\label{fig3ex}}
\end{figure}

\end{section}

{\bf Acknowledgements.} The authors would like to thank Val\'erie Berth\'e, Anne Siegel, Wolfgang Steiner and Christiaan van de Woestijne for many valuable comments and discussions.

\bibliographystyle{abbrv}
\bibliography{nonunitPaper}
\end{document}